\documentclass[a4paper]{amsart}
\usepackage[utf8]{inputenc}
\usepackage[T1]{fontenc}
\usepackage{amscd,amsthm,amsfonts,amssymb,esint}
\usepackage{amsmath}
\usepackage[colorlinks=true]{hyperref}
\usepackage{fullpage}
\usepackage[all]{xy}
\usepackage{mathdots}
\usepackage{mathrsfs}
\usepackage{tikz-cd}
\usepackage{float}

    \newcommand{\BQ}{{\mathbb {Q}}}

     \newcommand{\BZ}{{\mathbb {Z}}}

    \newcommand{\CO}{{\mathcal {O}}} \newcommand{\CP}{{\mathcal {P}}}
     \newcommand{\CR}{{\mathcal {R}}}
    \newcommand{\CS}{{\mathcal {S}}} \newcommand{\CT}{{\mathcal {T}}}

     \newcommand{\fp}{{\mathfrak{p}}}

    \newcommand{\der}{{\mathrm{der}}}
    \newcommand{\diag}{{\mathrm{diag}}}

     \newcommand{\GL}{{\mathrm{GL}}}
    
    \newcommand{\Hom}{{\mathrm{Hom}}}

    \newcommand{\SL}{{\mathrm{SL}}}
     
    \newcommand{\SO}{{\mathrm{SO}}}\newcommand{\Sp}{{\mathrm{Sp}}}

\newcommand{\matrixx}[4]{\begin{pmatrix}
#1 & #2 \\ #3 & #4
\end{pmatrix} }        

    \newcommand{\sk}{\medskip}
    \newcommand{\lra}{\longrightarrow}
    \newcommand{\ra}{\rightarrow} 
    \newcommand{\bs}{\backslash}
    
    \newcommand{\s}{\sk\noindent}

    \theoremstyle{plain}
    \newtheorem{thm}{Theorem}[section] \newtheorem{cor}[thm]{Corollary}
    \newtheorem{lem}[thm]{Lemma}  \newtheorem{prop}[thm]{Proposition}
     \newtheorem{defn}[thm]{Definition}

\theoremstyle{remark} \newtheorem{remark}[thm]{Remark}
\theoremstyle{remark} 
\theoremstyle{remark} \newtheorem{example}[thm]{Example}

    \numberwithin{equation}{section}

     \newcommand{\Kl}{{\mathrm{Kl}}}

\begin{document}

\title{Stevens' Approach to Kloosterman Sums on Split Groups}

\author{Li Cai} 
\address[L.C.]{Academy for Multidisciplinary Studies, Beijing National Center for Applied Mathematics, Capital Normal University, Beijing, 100048, People's Republic of China
} 
\email{caili@cnu.edu.cn}

\author{Haoran Liu}
\address[H.L.]{Department of Mathematics Science, Tsinghua University, Beijing, 100084, People's Republic of China}
\email{1138909870@qq.com}

\author{Mingyu Yang}
\address[M.Y.]{School of Mathematics, Nanjing University, Jiangsu, 210093, People's Republic of China}
\email{1010850347@qq.com}

\author{Wenyan Yang}
\address[W.Y.]{School of Mathematical Sciences, University of Scinece and Technology of China, 230026, Hefei, P. R. China}
\email{ywy72@mail.ustc.edu.cn}

\begin{abstract}
We develop a uniform framework for  Stevens' approach to Kloosterman sums on split groups over $p$-adic fields. The main result reduces the bounding of Kloosterman sums to that for the so-called Plücker-constrained sets, which are certain subsets of affine spaces constrained by conditions giving by the Plücker coordinates. We then give a nontrivial bound for the Kloosterman sums
on the exceptional group $G_2$. As an application, we obtain the regularity of Bessel distributions associated with  generic representations on $G_2$.
\end{abstract}

\subjclass[2020]{Primary:11L05, 22E50; Secondary: 22E35, 11F70}

\keywords{
Whittaker models, Bessel distributions, Kloosterman sums}

\maketitle
\tableofcontents

\section{Introduction}\label{sec:intro}

\subsubsection{Kloosterman sums}

Kloosterman sums play an important role in analytic number theory and the theory of automorphic forms. They arise naturally in the Fourier expansions of Poincaré series and on the geometric side of the Kuznetsov trace formula. A central problem is to obtain nontrivial bounds for Kloosterman sums.

Let $F$ be a $p$-adic field with ring of integers $\mathcal O$, maximal ideal $\mathfrak p$, residue field cardinality $q=|\mathcal O/\mathfrak p|$. We also fix a uniformizer $\varpi$ of $F$.

Let $G$ be a split reductive group over $F$. Let
$B = AN$ be a Borel subgroup of $G$ with $A$ a maximal torus of $G$ and
$N$ the unipotent radical of $B$. Denote by $\delta$ the modulus character of $B$. Let $\psi$ be a generic character on $N$. Fix a split reductive $\mathcal O$-model of $G$ compatible with $B=AN$, and choose a representative $\dot w_0\in N_G(A)(\mathcal O)$ of the longest Weyl element. All notation involving $\mathcal O$-points refers to this model.

For an open compact subgroup $K$ of $G$ such that $\psi|_{N \cap K} = 1$, we consider the following Kloosterman sum 
\[\Kl(t) =\Kl(t;\psi,\dot{w_0},K) = \sum_{x \in X_K(\dot{w_0}t)} \psi(u(x))\psi(u'(x)), \quad t \in A.\]
Here, the Kloosterman set is
\[X_K(\dot{w_0}t) = (N\cap K)\bs \left(N\dot{w_0}t N \cap K \right)/(N\cap K)\]
and the two maps are
\[u: X_K(\dot{w_0}t) \ra (N\cap K)\bs N, \quad u': X_K(\dot{w_0}t) \ra N/(N\cap K).\]
In the above, if $[n_1\dot{w_0}tn_2] \in X_K(\dot{w_0}t)$ with $n_1, n_2 \in N$ and $t \in A$, then 
\[u([n_1\dot{w_0}tn_2]) = [n_1], \quad u'([n_1\dot{w_0}tn_2]) = [n_2].\]
	
We say that the Kloosterman sum $\mathrm{Kl}$ has a nontrivial bound (\cite[Definition 3.17]{CCL26}) if there exist constants $\varepsilon>0$ and $C>0$ (possibly depending on $p$) such that
\[\left|\mathrm{Kl}(t)\cdot\delta^{1/2-\varepsilon}(t)\right| \leq C, \quad t \in A \cap G^\der.\]

\begin{example}\label{exam-GL_2}
Consider the case $G = \SL_2$ with $F = \BQ_p$. Take $\psi\left[ \matrixx{1}{x}{0}{1} \right] = \psi_{\BQ_p}(x)$ 
where $\psi_{\BQ_p}(x) = e^{-2\pi i\iota(x)}$ and $\iota: \BQ_p/\BZ_p \hookrightarrow \BQ/\BZ$ is the natural embedding.
Let $\dot{w_0} = \matrixx{0}{-1}{1}{0}$ and let $K = G(\BZ_p)$ be a maximal compact open subgroup of $G$. For $a =
\matrixx{p^r}{0}{0}{p^{-r}} \in A$ with $r>0$, we have
$\delta(a) = p^{-2r}$ and
\[\mathrm{Kl}(a) = \mathrm{Kl}(r) = \sum_{x \in \BZ_p^\times/1+p^r\BZ_p} \psi_{\BQ_p}\left( \frac{x+x^{-1}}{p^r} \right), \quad r> 0.\]

Using the Riemann hypothesis for curves, Weil proves that
\cite{Wei48}
\[|\Kl(1)| \leq 2p^{1/2}.\]
For $r \geq 2$, Salié \cite{Sal32} expresses $\Kl(r)$ explicitly in terms of the cosine function and proves that
\[|\Kl(r)| \leq 2p^{r/2}.\]
\end{example}

Nontrivial bounds for Kloosterman sums have significant applications in both global and local aspects of the theory of automorphic representations. 

 Building on Petersson's work, Selberg developed the theory of Poincaré series on $\GL_2$ and used their Fourier expansions, involving Kloosterman sums, to study the Ramanujan conjecture \cite{Sel65}. The extension to higher rank was initiated by Bump, Friedberg, and Goldfeld, who studied Poincaré series and the associated Kloosterman sums for $\mathrm{GL}_3$ \cite{BFG86,BFG88}. Subsequently, Friedberg \cite{Fri86,Fri87} and Stevens \cite{Ste87} simultaneously and independently developed the theory for general $\mathrm{GL}_n$.

Cogdell, Li, Piatetski-Shapiro and Sarnak \cite{CLPS91} considered Poincaré series for $\mathrm{SO}(n,1)$ and studied the associated Kloosterman zeta functions. By applying Weil's bound to certain Kloosterman sums, they obtained estimates for the exceptional spectrum of congruence arithmetic hyperbolic manifolds. 

More recently, Assing and Blomer \cite{AB24} proved Sarnak's density conjecture for the principal congruence subgroups of $\mathrm{SL}_n(\mathbb{Z})$ of squarefree level. An essential ingredient in their proof is a nontrivial bound for Kloosterman sums on $\SL_n$ associated with a certain Weyl element (which is different from the longest Weyl element).

For local harmonic analysis, the regularity of
Bessel distributions for generic local representations is closely related to nontrivial bounds for Kloosterman sums \cite{CCL26}. The regularity for $G = \GL_n$ is established
by a nontrivial bound for Kloosterman sums on $\GL_n$ 
with respect to admissible compact open subgroups.

\subsubsection{Main results}

Stevens' approach \cite{Ste87} is a basic method for bounding Kloosterman
sums. Let $\{K_d\}_d$ be the family of principal congruence subgroups of $G$, and let $\{\Kl_d\}_d$ be the corresponding Kloosterman sums (Section \ref{sect:kloosterman}). The approach can be divided into two steps.

\s{\bf (Step 1)} By considering the action of $A_d = A \cap K_d$ on $K_d$
    for $d$ large enough, the Kloosterman sum for $K_d$ can be decomposed into
    a product of Kloosterman sums on $\GL_2$. Applying the Weil bound for
    these Kloosterman sums on $\GL_2$, the problem of bounding $\Kl_d$
    is reduced to
    a counting problem for the Kloosterman set. This step is given in 
    \cite{Ste87} for $G = \GL_n$ and \cite{CCL26} for general split groups (Theorem \ref{thm:fibre-implies-kl}).
    
\s{\bf (Step 2)} One then uses Pl\"ucker coordinates for the counting problem. This step has been studied by Stevens \cite{Ste87} for $\GL_3$, Man \cite{Man22,Man24} for $\Sp_4$, Huang \cite{GSW21} for $\GL_4$, Mujdei \cite{Muj24} for $\SO_6$ and Miao \cite{Mia24} for $\GL_n$ with $F = \BQ_p$.

In this paper, we give a uniform treatment of (Step 2) for general split
groups. The main result reduces the problem of bounding Kloosterman sums to counting points in
Pl\"ucker-constrained sets, which are subsets of affine spaces defined by conditions on Pl\"ucker coordinates.

We now state our main result concerning Stevens' approach. For this, we assume that $G$ is semisimple and simply connected. Denote by $\Phi^+$ the set of positive roots and by $\Delta$ the set of simple roots. Put $m=|\Delta|$. Let $\alpha_1,\ldots,\alpha_m$ be the simple roots and $\omega_1,\dots,\omega_m$ be the fundamental weights.

For each $i=1,\dots,m$, let $(\pi_{\omega_i},V_{\omega_i})$ be the irreducible representation
of $G$ of highest weight $\omega_i$. Then $\dim (V_{\omega_i}^*)^N = 1$.
Choose a nonzero vector $\ell_i\in(V_{\omega_i}^*)^N$ and
a basis $\{e_{ij}\}_{j=1}^{\dim V_{\omega_i}}$ of $V_{\omega_i}$ consisting of $A$-weight vectors.

For each $1 \leq i \leq m$ and  $1 \leq j \leq \dim V_{\omega_i}$, 
consider the Pl\"ucker coordinates
\[c_{ij}(g) =\left\langle\ell_i,\,\pi_{\omega_i}(g)e_{ij}\right\rangle, \quad g \in G.\]
Then $F[N\bs G]$ is generated by these $c_{ij}$ (Theorem \ref{thm:plucker-coordinates}).

We shall also consider the Pl\"ucker coordinate polynomials
$$p_{ij}(\mathbf x):=\left\langle\ell_i,\,\pi_{\omega_i}\bigl(\dot w_0\,n(\mathbf x)\bigr)e_{ij}\right\rangle,\qquad\mathbf x\in F^{\Phi^+}.$$
In the above, we fix an ordering $\Phi^+=\{\gamma_1,\dots,\gamma_{|\Phi^+|}\}$ by increasing height, and choose Chevalley root subgroup parametrizations $X_\gamma:\mathbb G_a\to N_\gamma$ compatible with the fixed split $\mathcal O$-model.
With respect to this ordering, every element in $N$ can be written uniquely as 
$$n(\mathbf x)=\prod_{\gamma\in\Phi^+}X_\gamma(x_\gamma),\qquad\mathbf x=(x_\gamma)_{\gamma\in\Phi^+}\in F^{\Phi^+}.$$
The Pl\"ucker coordinate polynomials are homogeneous (Theorem \ref{thm:plucker-homogeneity}).

\begin{defn}[Pl\"ucker-constrained set]\label{defn:plucker-constrained}
We fix a section $\CR \subset F$ of $F/\CO$ containing $0$.
We also choose polynomials $q_\gamma \in F[\mathbf x]$ for
each  $\gamma \in \Phi^+$ satisfying
\[q_{\alpha_i} = x_{\alpha_i}, \quad q_\gamma = x_\gamma + f_\gamma, \quad f_\gamma \in F[x_\beta: \beta < \gamma].\]
For $r=(r_1,\dots,r_m)\in\mathbb Z_{\ge0}^m$, define the \emph{Pl\"ucker-constrained set} 
$$S_r=\left\{\mathbf x=(x_\gamma)_\gamma\in F^{\Phi^+}\ \Big|\ x_{\alpha_i}=0,\ v(p_{ij}(\mathbf x))\ge -r_i \text{ for all } i,j, \ q_\gamma(\mathbf x) \in \CR \text{ for all } \gamma \in \Phi^+ \setminus \Delta \right\}.$$
\end{defn}
For each $r$, $S_r$ is a finite set (Corollary \ref{cor:finiteness}). The polynomials $f_\gamma$, $\gamma \in \Phi^+\setminus\Delta$, may be chosen arbitrarily subject to the condition above. This gives us the freedom to change coordinates in the counting problem.

\begin{thm}[Corollary \ref{cor:practical}]\label{main-thm}
Let $G$ be a split semisimple and simply connected group with the Cartan matrix $C_G$ satisfying $p\nmid \det(C_G)$. Assume there exist constants $\varepsilon>0$ and $C>0$ such that 
$$|S_r|\le C\, q^{(1-\varepsilon)\sum_{i=1}^m r_i}, \quad r \in \BZ_{\geq 0}^m.$$
Then for all sufficiently large $d$, the Kloosterman sum $\mathrm{Kl}_d$ has a nontrivial bound.
\end{thm}

\begin{example}\label{exam:SL3}
Let $G=\SL_3$, and let
$$\dot{w_0}=\begin{pmatrix}
0 & 0 & 1 \\
0 & -1 & 0 \\
1 & 0 & 0
\end{pmatrix}.$$
Let $B=AN$ be the Borel subgroup of upper triangular matrices, with diagonal torus $A$ and unipotent radical $N$. Thus
$$A=\{\diag(a,b,c)\mid a,b,c\in F^\times,\ abc=1\}.$$
Let $\Delta=\{\alpha_1,\alpha_2\}$ be the set of simple roots. The positive roots are
$$\Phi^+=\{\alpha_1,\alpha_2,\alpha_1+\alpha_2\}, \quad
\alpha_1(\diag(a,b,c))=ab^{-1},\qquad
\alpha_2(\diag(a,b,c))=bc^{-1},$$
so that the coroots are
$$\alpha_1^\vee(t)=\diag(t,t^{-1},1),\qquad
\alpha_2^\vee(t)=\diag(1,t,t^{-1}),$$
and the fundamental weights are
$$\omega_1(\diag(a,b,c))=a,\qquad \omega_2(\diag(a,b,c))=ab .$$

Let $V_1,V_2$ be the fundamental representations with highest weights
$\omega_1,\omega_2$, respectively. Here $V_1$ is the $3$-dimensional standard
representation and $V_2=V_1^*$ is its dual. Let $e_1,e_2,e_3$ be the standard
basis of $V_1$. The highest weight vector of $V_1^*$ is $\ell_1:=e_3^\vee$, hence
we obtain three Pl\"ucker coordinates
$$x_j(g):=\langle ge_j,e_3^\vee\rangle=g_{3j},\qquad 1\le j\le 3 .$$
For $V_2=V_1^*$ the highest weight vector of $V_2^*=V_1$ is $\ell_2:=e_1$, and
the matrix coefficients of $V_2$ are
$$y_j(g):=\langle e_1,g\cdot e_j^\vee\rangle=(g^{-1})_{j1},\qquad 1\le j\le 3 .$$
Hence we obtain $6$ Pl\"ucker coordinate polynomials:
$$p_{1j}(\mathbf x)=x_j(\dot w_0 n(\mathbf x)),\quad p_{2j}(\mathbf x)=y_j(\dot w_0 n(\mathbf x)).$$

We give a coordinate expression for $N$. The root subgroup parametrizations are
$$X_{\alpha_1}(t)=I+tE_{12},\qquad X_{\alpha_2}(t)=I+tE_{23},\qquad
X_{\alpha_1+\alpha_2}(t)=I+tE_{13},$$
so that, multiplying in increasing height,
$$\prod_{\gamma\in\Phi^+}X_\gamma(x_\gamma)=
\begin{pmatrix}
1 & x_{\alpha_1} & x_{\alpha_1}x_{\alpha_2}+x_{\alpha_1+\alpha_2}\\
0 & 1 & x_{\alpha_2}\\
0 & 0 & 1
\end{pmatrix}.$$
For the counting argument, it is convenient to make the following change of coordinates:
$$x=x_{\alpha_1},\qquad y=x_{\alpha_2},\qquad
z=x_{\alpha_1+\alpha_2}+x_{\alpha_1}x_{\alpha_2}.$$
The three polynomials $x,y,z$ are exactly
the polynomials $q_\gamma$. Each of them has the form $x_\gamma+f_\gamma$ with
$f_\gamma\in F[x_\beta:\beta<\gamma]$.

We have the Pl\"ucker coordinate polynomials
$$p_{11}=1,\quad p_{12}=x,\quad p_{13}=z,\qquad
p_{21}=xy-z,\quad p_{22}=-y,\quad p_{23}=1 .$$

Writing $r=(r_1,r_2)$, the Pl\"ucker-constrained set $S_r = S_{r_1,r_2}$ is
$$\begin{aligned}
S_{r_1,r_2}
&=\Bigl\{ (x_{\alpha_1},x_{\alpha_2},x_{\alpha_1+\alpha_2})\ \Bigm|\
x_{\alpha_1}=x_{\alpha_2}=0,\ q_{\alpha_1+\alpha_2}\in\CR,\ v(p_{ij})\ge -r_i\Bigr\}\\
&\stackrel{\sim}{=} \Bigl\{(x,y,z)\ \Bigm|\ x=y=0,\quad z\in\CR,\quad v(p_{ij})\ge -r_i\Bigr\}\\
&\stackrel{\sim}{=} \Bigl\{z\in\CR\ \Bigm|\ v(z)\ge -r_1,\quad v(z)\ge -r_2\Bigr\}.
\end{aligned}$$
Since the conditions are equivalent to $v(z)\ge-\min\{r_1,r_2\}$, and
since $\CR$ meets every coset of $\CO$ in exactly one point, we get
$$|S_{r_1,r_2}|=q^{\min\{r_1,r_2\}}\le q^{\frac12(r_1+r_2)} .$$
Since $\det(C_{\SL_3})=3$, Theorem~\ref{main-thm} therefore shows that when $p\neq 3$, the Kloosterman sums for
$\SL_3$ satisfy a nontrivial bound.
\end{example}

By the definition of the above Pl\"ucker-constrained sets, the difficulty for the counting is measured by the number of non-simple positive roots $\Phi^+ \setminus \Delta$.
\[
\begin{array}{c|cccccc}
 & \mathrm{SL}_2 & \mathrm{SL}_3 & \mathrm{Sp}_4 & \mathrm{SL}_4 &\SO_6 & G_2\\
\hline
\#(\Phi^+\setminus\Delta) & 0&1&2&3&3&4
\end{array}
\]

In this paper, we obtain a power saving for Kloosterman sums on $G_2$
(Theorem \ref{thm:G2-power}). The result has the following application in local harmonic analysis.

Let $\psi = \psi_0$ be the generic character on $N$  in \cite[Theorem 1.2]{CCL26}. Let $\pi$ be an irreducible smooth admissible representation of $G$ which
is generic in the sense that $\Hom_N(\pi,\psi) \not=0$. Let $\pi^\vee$ be the  contragredient representation of $\pi$. Then $\pi^\vee$ is also generic. Let $\ell_1  \in \Hom_N(\pi,\psi)$ and $\ell_2 \in \Hom_N(\pi^\vee,\psi^{-1})$ be nonzero Whittaker functionals on $\pi$ and
$\pi^\vee$. For any Schwartz function $f \in \CS(G)$, $\pi^\vee(f)\ell_2 \in (\pi^\vee)^\vee = \pi$. Consider the  Bessel distribution for $\pi$ 
\[B_\pi(f) = \ell_1 \left( \pi^\vee(f) \ell_2 \right), \quad f \in \CS(G).\]

Applying
\cite[Theorem 1.2]{CCL26} relating nontrivial bounds for Kloosterman sums to the regularity of Bessel distributions, we obtain

\begin{thm}[Theorem \ref{thm:G2-Bessel}]\label{thm:G2-Bessel-intro}
   Let $\pi$ be an irreducible smooth admissible generic representation of $G = G_2$ (with respect to $\psi$). Then the Bessel distribution for $\pi$ is regular: there is a unique smooth function $j_\pi$ defined on a dense subset $\Omega$ of $G$, locally integrable on $G$, such that
   \[B_\pi(f) = \int_G j_\pi(g)f(g)dg, \quad f \in \CS(G).\]
\end{thm}

\subsubsection{The proof}

We sketch the proof of Theorem \ref{main-thm}.
Fix $t=t_0\prod_{i=1}^m\alpha_i^\vee(\varpi^{r_i})\in A$ with $t_0\in A(\mathcal O)$ and $r=(r_i)_i\in\mathbb Z_{\ge0}^m$.
By Step 1 of Stevens' approach (Theorem \ref{thm:fibre-implies-kl}), the problem of bounding the Kloosterman sums is reduced to a counting
problem for the Kloosterman sets $X(\dot{w_0}t)_y$ with $y \in \BZ^m$
consisting of $x \in X(\dot{w_0}t)$ such that the valuations of 
$u'(x)_\alpha$, $\alpha\in \Delta$ are given by $y$. 

Using the Pl\"ucker coordinates, for each $y$,  $|X(\dot{w_0}t)_y|$
is bounded by the cardinality of the following quotient set (Lemma \ref{lem:X-fibre-bound}):
\[\CT_{y,r}/\sim, \quad r \in \BZ_{\geq 0}^m.\]
Here, with $\nu(x):=\min\{v(x),0\}$, for each $r$, 
$$\mathcal T_{y,r}=\left\{\mathbf x=(x_\gamma)_\gamma\in F^{\Phi^+}\ \Big|\ \nu(x_{\alpha_i})=y_i,\quad v(p_{ij}(\mathbf x))\ge -r_i\ \text{for all }i,j\right\}$$
and we equip this set with the equivalence relation
$$\mathbf x\sim \mathbf x'\iff n(\mathbf x)N_0 = n(\mathbf x')N_0, \quad N_0 = N \cap K_0.$$

The key point in the proof of Theorem \ref{main-thm} is that it suffices to count in the case $y=0$.
More precisely, by the homogeneity of the Pl\"ucker polynomials, there is a map (Lemma \ref{lem:Tyr-to-T0r})
$$\mathcal T_{y,r}/{\sim}  \lra  \mathcal T_{0,r}/{\sim}$$
given by
$$(x_\gamma)_\gamma\mapsto\Bigl(x_\gamma\prod_{1 \leq i \leq m}
\varpi^{-y_i\kappa_{\alpha_i}(\gamma)}\Bigr)_\gamma$$
where for each $\gamma \in \Phi^+$, we write 
$\gamma=\sum_{1 \leq i \leq m }\kappa_{\alpha_i}(\gamma)\alpha_i$.
Moreover, every fibre of the above map is sufficiently small: its cardinality is at most
$q^{-\sum_{i=1}^m \rho(\alpha_i)y_i}$ with $\rho(\alpha_i) = \sum_{\gamma \in \Phi^+} \kappa_{\alpha_i}(\gamma)$. This is the reason the above map is called a perturbation map.
This bound on the fibres reduces the counting of Kloosterman sets to that of the set
$\CT_{0,r}/{\sim}$. 

A system of representatives for the equivalence classes can be given as follows. For a fixed section $\CR$ and polynomials $(q_\gamma)_\gamma$ as in Definition \ref{defn:plucker-constrained}, there is
a bijection (Lemma \ref{lem:normal-form})
\[\left\{\mathbf x=(x_\gamma)_\gamma\in F^{\Phi^+}\ \Big|\ q_\gamma(\mathbf x)\in\mathcal R\text{ for all }\gamma\in\Phi^+\right\} \stackrel{\sim}{\lra} N/N_0.\]
This completes the proof of Theorem \ref{main-thm}.

We discuss the counting of  Pl\"ucker-constrained sets for
$G = G_2$. Let $\Delta=\{\alpha,\beta\}$ be the simple roots, with $\alpha$ the
short root and $\beta$ the long one. The positive roots are
$\Phi^+=\{\alpha,\beta,\alpha+\beta,2\alpha+\beta,3\alpha+\beta,3\alpha+2\beta\}$.
Denote by $\pi_1$ and $\pi_2$ the irreducible representations
of $G$ with highest weights $\omega_1 = 2\alpha+\beta$ and $\omega_2 = 3\alpha+2\beta$, respectively. Then $\pi_1$
is the $7$-dimensional standard representation. The representation 
$\pi_2$ is the $14$-dimensional adjoint representation of $G$, and
\[\wedge^2 \pi_1 = \pi_1 \oplus \pi_2.\]
In particular, we have $21$ Pl\"ucker coordinates for $G$.

For $\mathbf r=(r,s)\in\mathbb Z_{\ge 0}^2$, write $m:=\min\{r,s\}$.  The Pl\"ucker-constrained set $S_{r,s} = S_{\mathbf r}$ in
Theorem~\ref{main-thm} is (Section \ref{sec:G2prep})
$$S_{r,s}\stackrel{\sim}{=} \left\{(c,d,e,f)\in\CR^4\ \middle|\ 
\begin{aligned}
&v(c),\,v(d),\,v(e),\,v(f),\,v(d^2-ce)\ \ge\ -m,\\
&v(cd),\,v(c^3),\,v(c^2d-cf),\,v(c^2e-df)\ \ge\ -s,\\
&v(d^3-2cde+ef),\,v(c^2d^2-c^3e-cdf+f^2)\ \ge\ -s
\end{aligned}
\right\}.$$

For each $x \in F$, we call $\max\{0,-v(x)\}$ the depth of $x$. We count $S_{r,s}$ by fixing the depths $L_c,L_d,L_e$ of $c,d,e$ and counting in the order $c,d,e,f$:
\begin{itemize}
    \item For fixed positive depths $L_c,L_d,L_e \geq 1$, Proposition~\ref{prop:positive-depth}
gives a bound of $\ll q^{\frac{21}{22}(r+s)}$ for the number of points.
\item For fixed depths with $L_cL_dL_e = 0$, Proposition~\ref{prop:zero-depth} gives a bound of $\ll q^{\frac56(r+s)}$ for the number of points.
\end{itemize}
These propositions give the estimate (Corollary~\ref{cor:G2-count})
$$|S_{r,s}|\ll(1+r+s)^3q^{\frac{21}{22}(r+s)}.$$
By Theorem~\ref{main-thm}, this gives a nontrivial bound for
Kloosterman sums on $G$ (Theorem~\ref{thm:G2-power}). 

\s{\bf Acknowledgements.} This work was started during the Algebra and Number Theory Summer School held at Peking University from July 13 to August 21, 2026. The authors are grateful to Shouwu Zhang and Liang Xiao for organizing the summer school, and to Peking University for its hospitality and stimulating research environment. L. Cai is supported by the National Key R\&D Program of China No.
2023YFA1009702 and National Natural Science Foundation of
China, No. 12371012. 

\s{\bf Use of artificial intelligence.} ChatGPT and DeepSeek tools were used for exploratory computations, editing, searching for references and for verification of results.

\section{Pl\"ucker coordinates}

\subsection{Highest weight theory and Peter--Weyl theorem}

In this subsection we recall the standard representation-theoretic facts that will be used later.

Let $G$ be a split semisimple and simply connected algebraic group over $F$, and let $B=AN$ be a Borel subgroup. We write $V^*$ for the contragredient of a representation $V$ of $G$.

Recall that
$$X^*(A)=\operatorname{Hom}_{F\text{-grp}}(A,\mathbb G_m),\qquad
X_*(A)=\operatorname{Hom}_{F\text{-grp}}(\mathbb G_m,A),$$
with the natural pairing $\langle\cdot,\cdot\rangle:
X^*(A)\times X_*(A)\longrightarrow \mathbb Z$. We use additive notation for characters.

Let $\Phi\subset X^*(A)$ be the root system, let $\Phi^+$ be the set of positive roots, and let $\Delta=\{\alpha_1,\ldots,\alpha_m\}$ be the simple roots. For each $i$, let $\alpha_i^\vee\in X_*(A)$ be the corresponding simple coroot.

For two weights $\lambda,\mu\in X^*(A)$, we write $\mu\le\lambda$ if
$$\lambda-\mu=\sum_{i=1}^m n_i\alpha_i,\qquad n_i\in\mathbb Z_{\ge 0}.$$
We write $\mu<\lambda$ if $\mu\le\lambda$ and $\mu\ne\lambda$.

A weight $\lambda\in X^*(A)$ is called dominant if $\langle\lambda,\alpha_i^\vee\rangle\ge 0$ for every $i$. The monoid of dominant weights is
$$X^+(A)=\left\{\lambda\in X^*(A):\langle\lambda,\alpha_i^\vee\rangle\ge 0\text{ for all }i\right\}.$$
Since $G$ is semisimple and simply connected, there exist fundamental weights $\omega_1,\ldots,\omega_m\in X^*(A)$ characterized by $\langle\omega_i,\alpha_j^\vee\rangle=\delta_{ij}$. Every dominant weight admits a unique expression
$$\lambda=\sum_{i=1}^m n_i\omega_i,\qquad n_i\in\mathbb Z_{\ge 0}.$$

If $V$ is a finite-dimensional representation of $G$, then its restriction to $A$ has a weight-space decomposition
$$V=\bigoplus_{\chi\in X^*(A)}V[\chi],\qquad V[\chi]=\{v\in V:t\cdot v=\chi(t)v\text{ for every }t\in A\}.$$
A nonzero vector in $V[\chi]$ is called a weight vector of weight $\chi$.

\begin{thm}[{Highest weight theory; \cite[Chapter VII]{Jac79}}]\label{thm:highest-weight}
Let $V$ be a finite-dimensional irreducible representation of $G$. Then there exists a unique dominant weight $\lambda\in X^+(A)$ such that:
\begin{enumerate}
\item $\lambda$ is a weight of $V$, and every weight $\chi$ of $V$ satisfies $\chi\le\lambda$;

\item the highest weight space $V[\lambda]$ is one-dimensional and equals the space of $N$-invariant vectors
$$V[\lambda]=V^N=\{v\in V:n\cdot v=v\text{ for every }n\in N\};$$

\item $V$ is generated by any nonzero vector in $V[\lambda]$.
\end{enumerate}
Conversely, for every $\lambda\in X^+(A)$, there exists up to isomorphism a unique finite-dimensional irreducible representation $V_\lambda$ of highest weight $\lambda$. Thus the map $\lambda\mapsto [V_\lambda]$ is a bijection between $X^+(A)$ and the isomorphism classes of finite-dimensional irreducible representations of $G$.
\end{thm}

Let $W=N_G(A)/A$ be the Weyl group, and let $w_0\in W$ be its longest element. For $\lambda\in X^+(A)$, the contragredient representation $V_\lambda^*$ is irreducible of highest weight $\lambda^*=-w_0\lambda$.

\begin{thm}[{Cartan component; \cite[Chapter VII]{Jac79}}]\label{thm:cartan-component}
For $\lambda,\mu\in X^+(A)$, the tensor product $V_\lambda\otimes_F V_\mu$ decomposes as
$$V_\lambda\otimes_F V_\mu\simeq V_{\lambda+\mu}\oplus\bigoplus_{\substack{\nu\in X^+(A)\\ \nu<\lambda+\mu}}V_\nu^{\oplus m_\nu},\qquad m_\nu\in\mathbb Z_{\ge 0}.$$
In particular, $V_{\lambda+\mu}$ occurs in $V_\lambda\otimes_F V_\mu$ with multiplicity one, and every nonzero $G$-equivariant map
$$V_\lambda\otimes_F V_\mu\longrightarrow V_{\lambda+\mu}$$
is surjective. 
\end{thm}

\begin{thm}[{Algebraic Peter--Weyl; \cite{TY05}, \cite[Corollary 3.1.12]{Hay18}}]\label{thm:algebraic-pw}
Let $F[G]$ be the coordinate ring of $G$, equipped with the $G\times G$-action
$$((g_1,g_2)\cdot f)(x)=f(g_1^{-1}xg_2),\qquad g_1,g_2,x\in G.$$
For $\lambda\in X^+(A)$, let $\pi_\lambda:G\to \operatorname{GL}(V_\lambda)$
be the irreducible representation of highest weight $\lambda$. For $\ell\in V_\lambda^*$ and $v\in V_\lambda$, consider the matrix coefficient
$$c_{\ell,v}(g)=\langle \ell,\pi_\lambda(g)v\rangle.$$
Then the matrix coefficient maps induce an isomorphism of rational $G\times G$-modules
$$\bigoplus_{\lambda\in X^+(A)}V_\lambda^*\otimes_F V_\lambda\xrightarrow{\ \sim\ }F[G],\qquad\ell\otimes v\longmapsto c_{\ell,v}.$$
Under this isomorphism, the left copy of $G$ acts on $V_\lambda^*$ and the right copy of $G$ acts on $V_\lambda$.
\end{thm}

\subsection{Pl\"ucker coordinates}

Let
$$F[N\backslash G]:=F[G]^N=\{f\in F[G]:f(ng)=f(g)\text{ for all }n\in N,\ g\in G\}.$$
Taking $N$-invariants on the left in the algebraic Peter--Weyl decomposition
(Theorem~\ref{thm:algebraic-pw}) gives an isomorphism of right $G$-modules
$$F[N\backslash G]\simeq\bigoplus_{\lambda\in X^+(A)}(V_\lambda^*)^N\otimes_F V_\lambda.$$
By highest weight theory (Theorem~\ref{thm:highest-weight}), each space
$(V_\lambda^*)^N$ is one-dimensional.

For every $\lambda\in X^+(A)$, choose a nonzero vector $\ell_\lambda\in (V_\lambda^*)^N$.
For $v\in V_\lambda$, set the matrix coefficient
$$c_{\lambda,v}(g):=\langle \ell_\lambda,\pi_\lambda(g)v\rangle,\qquad g\in G.$$
Since $\ell_\lambda$ is fixed by $N$, the function $c_{\lambda,v}$ is left
$N$-invariant, hence belongs to $F[N\backslash G]$.

Define a right $G$-module homomorphism
$$\Theta_\lambda:V_\lambda\longrightarrow F[N\backslash G],\qquad\Theta_\lambda(v):=c_{\lambda,v},$$
and set $U_\lambda:=\Theta_\lambda(V_\lambda)$.
Consequently, the algebraic Peter--Weyl decomposition induces an isomorphism
of right $G$-modules
$$\bigoplus_{\lambda\in X^+(A)}V_\lambda\xrightarrow{\ \sim\ }F[N\backslash G],\qquad v\longmapsto c_{\lambda,v}.$$

\begin{thm}\label{thm:plucker-coordinates}The $F$-algebra $F[N\backslash G]$ is generated by
$U_{\omega_1},\ldots,U_{\omega_m}$. Equivalently, for each
$i=1,\ldots,m$, choose a basis $\{e_{ij}\}_{j}$ of $V_{\omega_i}$ and a
nonzero vector $\ell_i\in (V_{\omega_i}^*)^N$, and set
$$c_{ij}(g):=\langle \ell_i,\pi_{\omega_i}(g)e_{ij}\rangle.$$
Then
$$F[N\backslash G]=F\bigl[c_{ij}:1\le i\le m,\ 1\le j\le \dim V_{\omega_i}\bigr].$$
\end{thm}

\begin{proof}

Consider the left action of $A$ on $F[N\backslash G]$ defined
by $(L_t f)(g):=f(t^{-1}g)$.
If $f=c_{\lambda,v}\in U_\lambda$, then
$$(L_t f)(g)=\langle \ell_\lambda,\pi_\lambda(t^{-1}g)v\rangle=\langle \pi_\lambda^*(t)\ell_\lambda,\pi_\lambda(g)v\rangle=\lambda^*(t)c_{\lambda,v}(g),$$
where $\lambda^*=-w_0\lambda$. Thus $U_\lambda$ is the $\lambda^*$-weight
space for the left $A$-action. Since
$(\lambda+\mu)^*=\lambda^*+\mu^*$, pointwise multiplication gives
$$U_\lambda U_\mu\subset U_{\lambda+\mu}.$$

Multiplication therefore defines a homomorphism of right $G$-modules
$$m_{\lambda,\mu}:U_\lambda\otimes_F U_\mu\longrightarrow U_{\lambda+\mu},\qquad f\otimes f'\longmapsto ff'.$$
Since $G$ is connected, $F[G]$ is an integral domain, and so is its
subalgebra $F[N\backslash G]$. Hence $m_{\lambda,\mu}\neq 0$. Using the right
$G$-module identifications $U_\lambda\simeq V_\lambda$ and
$U_\mu\simeq V_\mu$, the map $m_{\lambda,\mu}$ becomes a nonzero
$G$-equivariant map $V_\lambda\otimes_F V_\mu\longrightarrow V_{\lambda+\mu}$.
By the Cartan component theorem (Theorem~\ref{thm:cartan-component}), every
nonzero such map is surjective. Therefore
$$U_\lambda U_\mu=U_{\lambda+\mu}.$$

Now let $\lambda\in X^+(A)$ and write $\lambda=\sum_{i=1}^m n_i\omega_i,n_i\ge 0$.
Repeated application of $U_\lambda U_\mu=U_{\lambda+\mu}$ gives
$$U_\lambda\subset F[U_{\omega_1},\ldots,U_{\omega_m}].$$
Together with the direct sum decomposition
$$F[N\backslash G]=\bigoplus_{\lambda\in X^+(A)}U_\lambda,$$
this proves that $F[N\backslash G]$ is generated by
$U_{\omega_1},\ldots,U_{\omega_m}$, and hence by the functions $c_{ij}$.
\end{proof}
\begin{remark}
The case $G=\SL_n$ is classical: the $c_{ij}$ are then the Pl\"ucker
coordinates of the flag determined by $g$, and the assertion of
Theorem~\ref{thm:plucker-coordinates} is the corresponding generation statement
for the coordinate ring of the basic affine space; see \cite{Dol03}, \cite{Ses16} for the
classical theory.
\end{remark}

\subsection{Homogeneity of Pl\"ucker coordinates}

Let $G$ be a split semisimple and simply connected algebraic group over $F$, and let $B=AN$ be a Borel subgroup. Fix an ordering of the positive roots $\Phi^+$, and choose root subgroup parametrizations
$$X_\gamma:\mathbb G_a\longrightarrow N_\gamma,\qquad \gamma\in\Phi^+.$$
Using the fixed ordering, every $n\in N$ can be written uniquely as
$$n(\mathbf x)=\prod_{\gamma\in\Phi^+}X_\gamma(x_\gamma),\qquad \mathbf x=(x_\gamma)_{\gamma\in\Phi^+}.$$

Let $Q^+:=\sum_{\alpha\in \Delta}\mathbb Z_{\ge 0}\alpha$ be the positive root monoid.
We say that a polynomial $P\in F[\mathbf x]$ is \emph{$\mu$-homogeneous} for some $\mu\in Q^+$ if, writing
$$P=\sum_{\mathbf k}b_{\mathbf k}\mathbf x^{\mathbf k},\qquad \mathbf k=(k_\gamma)_{\gamma\in\Phi^+},$$
every monomial with $b_{\mathbf k}\neq 0$ satisfies
$$\sum_{\gamma\in\Phi^+}k_\gamma\gamma=\mu.$$

\begin{lem}\label{lem:homogeneous-scaling}
Let $\mu\in Q^+$ and let $P\in F[\mathbf x]$. Then $P$ is $\mu$-homogeneous if and only if
$$P\bigl((\gamma(t)x_\gamma)_\gamma\bigr)=\mu(t)P(\mathbf x)\qquad\text{for all }t\in A.$$
\end{lem}

\begin{proof}
Write
$$P(\mathbf x)=\sum_{\mathbf k}b_{\mathbf k}\mathbf x^{\mathbf k},\qquad\mathbf k=(k_\gamma)_{\gamma\in\Phi^+}.$$
Then
$$P\bigl((\gamma(t)x_\gamma)_\gamma\bigr)=\sum_{\mathbf k}b_{\mathbf k}\prod_{\gamma\in\Phi^+}\gamma(t)^{k_\gamma}x_\gamma^{k_\gamma}=\sum_{\mathbf k}b_{\mathbf k}\Bigl(\sum_{\gamma\in\Phi^+}k_\gamma\gamma\Bigr)(t)\,\mathbf x^{\mathbf k}.$$
Thus the identity
$$P\bigl((\gamma(t)x_\gamma)_\gamma\bigr)=\mu(t)P(\mathbf x)\qquad\text{for all }t\in A$$
is equivalent to
$$\sum_{\mathbf k}b_{\mathbf k}\Bigl(\Bigl(\sum_{\gamma\in\Phi^+}k_\gamma\gamma\Bigr)(t)-\mu(t)\Bigr)\,\mathbf x^{\mathbf k}=0\qquad\text{for all }t\in A.$$
Since the monomials $\mathbf x^{\mathbf k}$ are linearly independent, this holds if and only if, for every $\mathbf k$ with $b_{\mathbf k}\neq 0$, the character $(\sum_{\gamma\in\Phi^+}k_\gamma\gamma)-\mu$ is identically one on $A$. Equivalently, $\sum_{\gamma\in\Phi^+}k_\gamma\gamma=\mu$ whenever $b_{\mathbf k}\neq 0$. This is exactly the definition of $\mu$-homogeneity.
\end{proof}

\begin{thm}[Homogeneity of Pl\"ucker coordinates]\label{thm:plucker-homogeneity}
Let $\lambda\in X^+(A)$, let $V_\lambda$ be the irreducible representation of highest weight $\lambda$, choose $0\neq\ell_\lambda\in (V_\lambda^*)^N$, and let $v\in V_\lambda[\chi]$ be a weight vector of weight $\chi$. For $t\in A$, define
$$P_{\lambda,v,t}(\mathbf x):=\bigl\langle\ell_\lambda,\pi_\lambda(\dot w_0 t\,n(\mathbf x))v\bigr\rangle.$$
Then $P_{\lambda,v,t}$ is a polynomial in the variables $(x_\gamma)_{\gamma\in\Phi^+}$, and it is $(\lambda-\chi)$-homogeneous.
\end{thm}

\begin{proof}
Since $\pi_\lambda:G\to\operatorname{GL}(V_\lambda)$ is algebraic and each root subgroup parametrization $X_\gamma$ is algebraic, the matrix entries of $\pi_\lambda(\dot w_0 t\,n(\mathbf x))$ are regular functions on $F^{|\Phi^+|}$. Hence $P_{\lambda,v,t}$ is a polynomial.

By the highest weight theorem (Theorem~\ref{thm:highest-weight}), every weight $\chi$ of $V_\lambda$ satisfies $\chi\le\lambda$, so $\lambda-\chi\in Q^+$.

We now prove the homogeneity statement. 
By Lemma~\ref{lem:homogeneous-scaling}, it suffices to prove that for all $s\in A$,
$$P_{\lambda,v,t}\bigl((\gamma(s) x_\gamma)_\gamma\bigr)=(\lambda-\chi)(s)\,P_{\lambda,v,t}({\mathbf x}).$$
Let $s\in A$. Conjugation by $s$ acts on root subgroups by $sX_\gamma(u)s^{-1}=X_\gamma(\gamma(s)u)$, and therefore 
$$s\,n({\mathbf x})s^{-1}=\prod_{\gamma\in\Phi^+}(sX_\gamma(x_\gamma)s^{-1})=\prod_{\gamma\in\Phi^+}X_\gamma(\gamma(s)x_\gamma)=n((\gamma(s) x_\gamma)_\gamma).$$
Thus
$$P_{\lambda,v,t}\bigl((\gamma(s) x_\gamma)_\gamma\bigr)=\bigl\langle\ell_\lambda,\pi_\lambda(\dot w_0 t\,n((\gamma(s) x_\gamma)_\gamma))v\bigr\rangle =\bigl\langle\ell_\lambda,\pi_\lambda(\dot w_0 t\,s\,n({\mathbf x})s^{-1})v\bigr\rangle.$$
Since $t$ and $s$ commute, $\dot w_0 t\,s=(\dot w_0 s\dot w_0^{-1})\,\dot w_0 t$.
Put $s^{w_0}=\dot w_0 s\dot w_0^{-1}$, and so
$$P_{\lambda,v,t}\bigl((\gamma(s) x_\gamma)_\gamma\bigr)=\bigl\langle\ell_\lambda,\pi_\lambda(s^{w_0})\pi_\lambda(\dot w_0 t\, n({\mathbf x}))\pi_\lambda(s^{-1})v\bigr\rangle.$$
Because $v$ has weight $\chi$, $\pi_\lambda(s^{-1})v=(-\chi)(s)v$.
Also, $\ell_\lambda$ has weight $-w_0\lambda$ in the contragredient representation, so for every $z\in V_\lambda$,
$$\bigl\langle \ell_\lambda,\pi_\lambda(s^{w_0})z\bigr\rangle=(-w_0\lambda)\bigl((s^{w_0})^{-1}\bigr)\bigl\langle \ell_\lambda,z\bigr\rangle=\lambda(s)\bigl\langle \ell_\lambda,z\bigr\rangle.$$
Hence,
$$P_{\lambda,v,t}\bigl((\gamma(s)x_\gamma)_\gamma\bigr)=\lambda(s)(-\chi)(s)\bigl\langle\ell_\lambda,\pi_\lambda(\dot w_0 t\, n({\mathbf x}))v\bigr\rangle=(\lambda-\chi)(s)P_{\lambda,v,t}(\mathbf x).$$
Since this holds for every $s\in A$, Lemma~\ref{lem:homogeneous-scaling} implies that $P_{\lambda,v,t}$ is $(\lambda-\chi)$-homogeneous.
\end{proof}

\begin{cor}\label{cor:plucker-torus}
Let $i=1,\ldots,m$, and let $\{e_{ij}\}_j$ be a basis of $V_{\omega_i}$ consisting of $A$-weight vectors, say $e_{ij}\in V_{\omega_i}[\chi_{ij}]$. Choose a nonzero vector $\ell_i\in (V_{\omega_i}^*)^N$, and define the Pl\"ucker coordinate polynomial
$$p_{ij}(\mathbf x):=\bigl\langle\ell_i,\pi_{\omega_i}(\dot w_0\,n(\mathbf x))e_{ij}\bigr\rangle.$$
For $t\in A$, set
$$p_{ij,t}(\mathbf x):=\bigl\langle\ell_i,\pi_{\omega_i}(\dot w_0 t\,n(\mathbf x))e_{ij}\bigr\rangle.$$
Then $p_{ij,t}(\mathbf x)=\omega_i(t)\,p_{ij}(\mathbf x)$.
Moreover, $p_{ij}$ is $(\omega_i-\chi_{ij})$-homogeneous.
\end{cor}

\begin{proof}
Since $\ell_i$ has weight $-w_0\omega_i$ in the contragredient representation, for any $s\in A$ and $z\in V_{\omega_i}$,
$$\bigl\langle \ell_i,\pi_{\omega_i}(s)z\bigr\rangle=(-w_0\omega_i)(s^{-1})\bigl\langle \ell_i,z\bigr\rangle.$$
Taking $s=\dot w_0 t\dot w_0^{-1}$ gives $(-w_0\omega_i)(\dot w_0 t^{-1}\dot w_0^{-1})=\omega_i(t)$,
and hence
$$\bigl\langle\ell_i,\pi_{\omega_i}(\dot w_0 t\dot w_0^{-1})z\bigr\rangle=\omega_i(t)\,\bigl\langle \ell_i,z\bigr\rangle.$$
Taking $z=\pi_{\omega_i}(\dot w_0 n(\mathbf x))e_{ij}$ yields
$$p_{ij,t}(\mathbf x)=\omega_i(t)\,p_{ij}(\mathbf x).$$
The homogeneity assertion follows immediately from Theorem~\ref{thm:plucker-homogeneity} with $\lambda=\omega_i$, $v=e_{ij}$, and $\chi=\chi_{ij}$.
\end{proof}

The preceding results show that each Pl\"ucker coordinate is homogeneous, but they do not yet relate the Pl\"ucker coordinates to the individual root coordinates $x_\gamma$.  The following lemma fills this gap; we will use it in the finiteness proof below.

\begin{lem}\label{lem:plucker-monomial}
For the Pl\"ucker coordinates $p_{ij}$ defined above, and for every $\gamma\in\Phi^+$, there exist $i$ and $j$ such that $p_{ij}$ is $\gamma$-homogeneous and
$$p_{ij}=c_\gamma x_\gamma+g_\gamma,\qquad
c_\gamma\in F^\times,\quad g_\gamma\in F[x_\beta:\beta<\gamma].$$
\end{lem}

\begin{proof}
Write $\gamma^\vee=\sum_{i=1}^m c_i(\gamma)\alpha_i^\vee\neq0$, and let $i=i(\gamma)$ be an index such that $c_i(\gamma)\neq0$. We claim that there exists an index $j$ such that $\omega_i-\chi_{ij}=\gamma$ and $p_{ij}(\mathbf x)$ has an $x_\gamma$-term with nonzero coefficient.

Suppose not. Then for every $j$ with $e_{ij}\in V_{\omega_i}[\omega_i-\gamma]$, the polynomial $p_{ij}(\mathbf x)$ does not have the monomial $x_\gamma$. Since such $e_{ij}$ span the weight space $V_{\omega_i}[\omega_i-\gamma]$, we obtain that for every $v\in V_{\omega_i}[\omega_i-\gamma]$, the polynomial
$$p_v(\mathbf x):=\langle \ell_i,\pi_{\omega_i}(\dot w_0\,n(\mathbf x))v\rangle$$
does not have the monomial $x_\gamma$. To obtain a contradiction, we construct a vector $v\in V_{\omega_i}[\omega_i-\gamma]$ such that $p_v$ has an $x_\gamma$-term.

Let $(e_\gamma,f_\gamma,h_\gamma)$ be the $\mathfrak{sl}_2$-triple associated with $\gamma$ \cite{Car72}, normalized by
$$[e_\gamma,f_\gamma]=h_\gamma=\gamma^\vee.$$
Let $v^+$ be a nonzero highest-weight vector of $V_{\omega_i}$. Then $e_\gamma v^+=0$. Set $v_\gamma=f_\gamma v^+$. Then
$$e_\gamma v_\gamma=e_\gamma f_\gamma v^+=[e_\gamma,f_\gamma]v^+=\gamma^\vee v^+=\langle\omega_i,\gamma^\vee\rangle v^+.$$
Since $\langle\omega_i,\gamma^\vee\rangle=c_i(\gamma)\neq0$, we have $v_\gamma\neq0$. Moreover, $v_\gamma$ has weight $\omega_i-\gamma$.

It remains to show that $p_{v_\gamma}$ has an $x_\gamma$-term. Set $x_\beta=0$ for all $\beta\neq\gamma$ and $x_\gamma=s$. Then $n(\mathbf x)=X_\gamma(s)$, and the coefficient of $x_\gamma$ in $p_{v_\gamma}$ is
$$\frac{\mathrm d}{\mathrm ds}\Big|_{s=0}\langle \ell_i,\pi_{\omega_i}(\dot w_0\,X_\gamma(s))v_\gamma\rangle=\langle \ell_i,\pi_{\omega_i}(\dot w_0)e_\gamma v_\gamma\rangle=\langle \ell_i,\pi_{\omega_i}(\dot w_0)v^+\rangle\cdot\langle\omega_i,\gamma^\vee\rangle.$$
Since $\ell_i$ is a nonzero $N$-invariant vector in $V_{\omega_i}^*$, the pairing $\langle \ell_i,\pi_{\omega_i}(\dot w_0)v^+\rangle$ is nonzero. Thus $p_{v_\gamma}$ has an $x_\gamma$-term. This contradiction proves the claim.

Let $j(\gamma)$ be an index such that $\omega_i-\chi_{i,j(\gamma)}=\gamma$ and $p_{i,j(\gamma)}(\mathbf x)$ has an $x_\gamma$-term. Then by Theorem \ref{thm:plucker-homogeneity}, $p_{i,j(\gamma)}$ is $\gamma$-homogeneous. Hence $p_{i,j(\gamma)}-c_\gamma x_\gamma$ is $\gamma$-homogeneous and does not have an $x_\gamma$-term. Therefore $p_{i,j(\gamma)}(\mathbf x)-c_\gamma x_\gamma$ lies in $F[x_\beta:\beta<\gamma]$. This completes the proof.
\end{proof}

\section{Stevens' approach}

In this section, we prove Theorem~\ref{main-thm} following the approach of Stevens \cite{Ste87}.

\subsection{Kloosterman sums}\label{sect:kloosterman}

We keep the notation of \S1: $G$ is a split semisimple simply connected group over $F$, $B=AN$ is a Borel subgroup, and $\psi$ is a fixed generic character of $N$. We use the fixed split Chevalley $\mathcal O$-model and compatible root subgroup parametrizations $X_\gamma$ for all $\gamma\in\Phi$, extending those already chosen for $\gamma\in\Phi^+$. Fix a representative $\dot w_0\in N_G(A)(\mathcal O)$ of the longest Weyl element. Let $\Delta=\{\alpha_1,\dots,\alpha_m\}$ and put $\Phi^{\vee,+}=\{\alpha^\vee:\alpha\in\Phi^+\}$.

Set $K_0=G(\mathcal O)$ and, for $d\ge1$, let $K_d$ be the kernel of the reduction map $G(\mathcal O)\to G(\mathcal O/\mathfrak p^d)$. Put $A_d=A\cap K_d$, $N_d=N\cap K_d$, and $\overline N_d=\overline N\cap K_d$. Then $B\cap K_d=A_dN_d$, and for $d\ge1$ we have
$$K_d=\overline N_dA_dN_d,\qquad A_d=\{a\in A:\chi(a)\in1+\mathfrak p^d\text{ for all }\chi\in X^*(A)\}.$$
Moreover, $\bigcap_{d\ge1}K_d=\{1\}$ and $N_d=\prod_{\gamma\in\Phi^+}X_\gamma(\mathfrak p^d)$ for every $d\ge0$.

For $d\ge1$ large enough so that $\psi|_{N_d}=1$ and for $t\in A$, set
$$X_d(\dot w_0t)=N_d\backslash\bigl(N\dot w_0tN\cap K_d\bigr)/N_d,$$
and for a class $x=[n_1\dot w_0tn_2]\in X_d(\dot w_0t)$ with $n_1,n_2\in N$, put
$$u(x)=[n_1]\in N_d\backslash N,\qquad u'(x)=[n_2]\in N/N_d .$$
The \emph{Kloosterman sums} $\{\Kl_d\}_{d\ge1}$ are the functions on $A$
defined by
$$\Kl_d(t)=\sum_{x\in X_d(\dot w_0t)}\psi\bigl(u(x)\bigr)\psi\bigl(u'(x)\bigr).$$
With the notation of \S1 this reads $\Kl_d(t)=\Kl(t;\psi,\dot w_0,K_d)$.

For $t\in A$, write $t=t_0\lambda_t(\varpi)$ with $t_0\in A(\mathcal O)$ and $\lambda_t\in X_*(A)$. By D\k{a}browski--Reeder \cite[Theorem 0.3]{DR98}, $X_d(\dot w_0t)$ is nonempty only if $\lambda_t\in\mathbb Z_{\ge0}\Phi^{\vee,+}$. Since $G$ is simply connected, it therefore suffices to consider
$$t=t_0\prod_{i=1}^m\alpha_i^\vee(\varpi^{r_i}),\qquad t_0\in A(\mathcal O),\quad r_i\in\mathbb Z_{\ge0}.$$

Let $\delta=\delta_B$ be the modulus character of $B$. By
\cite[Lemma 3.14]{CCL26} there is a constant $C_d>0$, independent of $t$, with
$$|\Kl_d(t)|\,\delta^{1/2}(t)\le C_d\Bigl(1+\sum_{i=1}^m r_i\Bigr)^{|\Phi^{\vee,+}|};$$
thus $|\Kl_d(t)\delta^{1/2}(t)|$ is bounded by a polynomial in the $r_i$.

We say that the system $\{\Kl_d\}_{d\ge1}$ has a \emph{nontrivial bound} if for
all sufficiently large $d$ there exist $\varepsilon>0$ and $C>0$ (possibly
depending on $d$) such that
$$\bigl|\Kl_d(t)\,\delta^{1/2-\varepsilon}(t)\bigr|\le C,\qquad t\in A .$$

\subsection{Stevens' trick}\label{sect:stevens_trick}

Throughout this section we assume that $p\nmid\det(C_G)$, where $C_G$ denotes
the Cartan matrix of $G$; this assumption is used only in the estimate of
\cite[Corollary 5.4]{CCL26} recalled below. By the support condition above, it suffices to consider
$$t=t_0\lambda_t(\varpi),\qquad t_0\in A(\mathcal O),\quad \lambda_t=\sum_{i=1}^m r_i\alpha_i^\vee,\quad r_i\in\mathbb Z_{\ge0}.$$

Using the fixed ordering and root subgroup parametrizations, write every $n\in N$ uniquely as
$$n(\mathbf x)=\prod_{\gamma\in\Phi^+}X_\gamma(x_\gamma),\qquad \mathbf x=(x_\gamma)_{\gamma\in\Phi^+}\in F^{\Phi^+}.$$
For $n\in N$, write $n_\gamma$ for its $\gamma$-coordinate. Recall that $N_d=\prod_{\gamma\in\Phi^+}X_\gamma(\mathfrak p^d)$ for $d\ge0$. For a positive integer $\ell$, put
$$N_{-\ell}:=\prod_{\gamma\in\Phi^+}X_\gamma\bigl(\mathfrak p^{-\ell\operatorname{ht}(\gamma)}\bigr),$$
where $\operatorname{ht}(\gamma)=\sum_{\alpha\in\Delta}\kappa_\alpha(\gamma)$ for $\gamma=\sum_{\alpha\in\Delta}\kappa_\alpha(\gamma)\alpha$.

\begin{prop}\label{prop:u_is_inj}
The maps $u$ and $u'$ are well-defined and injective.
\end{prop}

\begin{proof}
Because $\dot w_0$ is the longest Weyl element, we have
$N \cap \dot w_0t N (\dot w_0t)^{-1} = \{1\}$.  Consequently, every
$g \in N \dot w_0t N$ admits a unique decomposition
$g = n_{1}\dot w_0t n_{2}$ with $n_{1},n_{2}\in N$.

Let $x = N_{d} g N_{d}\in X_{d}(\dot w_0t)$ and write $g = n_{1}\dot w_0t n_{2}$.
Any other representative is of the form $h = \delta_{1}g\,\delta_{2}$
with $\delta_{1},\delta_{2}\in N_{d}$.  By uniqueness,
$h = (\delta_{1}n_{1})\dot w_0t (n_{2}\delta_{2})$. Hence
$u(h)=N_{d}\delta_{1}n_{1}=N_{d}n_{1}=u(g)$ and
$u'(h)=n_{2}\delta_{2}N_{d}=n_{2}N_{d}=u'(g)$. So $u, u'$ are well-defined.

Assume $u(x)=u(y)$ for $x=[g],\,y=[h]$ with
$g=n_{1}\dot w_0t n_{2},\,h=m_{1}\dot w_0t m_{2}\in K_{d}$.
From $N_{d}n_{1}=N_{d}m_{1}$ we obtain $\delta\in N_{d}$ with
$m_{1}=\delta n_{1}$.  Set $h'=\delta^{-1}h = n_{1}\dot w_0t m_{2}$.
Then $g^{-1}h'=n_{2}^{-1}m_{2}\in N\cap K_{d}=N_{d}$; thus
$m_{2}=n_{2}\delta'$ for some $\delta'\in N_{d}$.
Consequently $h = \delta g\delta'\in N_{d}g N_{d}$, so $x=y$, which implies that $u$ is injective.

The proof for $u'$ is symmetric.
\end{proof}

Choose $\ell>d$ such that for every $x\in X_d(\dot w_0 t)$ we have $u(x)\in N_d\backslash N_{-\ell}$ and $u'(x)\in N_{-\ell}/N_d$.
For $x=[n_1\dot w_0 t n_2]\in X_d(\dot w_0 t)$, let
$$\theta_i(x)=(n_1)_{\alpha_i}\in\mathfrak p^{-\ell}/\mathfrak p^d,\qquad\theta_i'(x)=(n_2)_{\alpha_i}\in\mathfrak p^{-\ell}/\mathfrak p^d.$$
Write $\theta'=(\theta_i')_{1\le i\le m}:X_d(\dot w_0 t)\to (\mathfrak p^{-\ell}/\mathfrak p^d)^m$ and set $\nu(x)=\min\{v(x),d\}$. 
For $y\in[-\ell,d]^m$, let $X_d(\dot w_0 t)_y$ be the fibre over $y$ under the composition
$$X_d(\dot w_0 t)\xrightarrow{\theta'}(\mathfrak p^{-\ell}/\mathfrak p^d)^m
\xrightarrow{\nu}[-\ell,d]^m.$$

The following lemma controls the growth of $\ell$. We will defer its proof to Section \ref{sect:finite}.

\begin{lem}\label{lem:ell-growth}
There exists a constant $C>0$, depending only on $G$ and $d$, such that for every $t\in A$ one can choose $\ell>d$ satisfying the above conditions and
$$\ell\le C\left(1+\sum_{i=1}^m r_i\right).$$
\end{lem}

We next record a bound for $\#X_d(\dot w_0 t)$.

\begin{prop}\label{prop:trivial-Xd}
There exists a constant $C_d>0$, independent of $t_0$ and $r$, such that for every $t=t_0\lambda_t(\varpi)$ with $t_0\in A(\mathcal O)$ and $\lambda_t=\sum_{i=1}^m r_i\alpha_i^\vee$, $r_i\in\mathbb Z_{\ge0}$, one has
$$|X_d(\dot w_0 t)|\le C_d\,q^{\sum_{i=1}^m r_i}\left(1+\sum_{i=1}^m r_i\right)^{|\Phi^{\vee,+}|}.$$
\end{prop}

\begin{proof}
Let $X(\dot w_0 t):=N_0\backslash\bigl(N \dot w_0 t N\cap K_0\bigr)/N_0$.
Since $K_d\subset K_0$, the natural map $X_d(\dot w_0 t)\to X(\dot w_0 t)$ has fibres of size at most $|N_0/N_d|^2<\infty$, which depends only on $d$. Hence $\#X_d(\dot w_0 t)\le C_d\#X(\dot w_0 t)$ for some $C_d>0$. It remains to bound the size of the set $X(\dot w_0t)$.

By \cite[Theorem 0.3]{DR98}, if $\lambda_t=\sum_{\gamma\in\Phi^{\vee,+}}n_\gamma\gamma$ with $n_\gamma\in\mathbb Z_{\ge0}$, then
$$|X(\dot w_0 t)|\asymp\delta^{-1/2}(t)\sum_{\substack{(n_\gamma)_\gamma\\
\sum_\gamma n_\gamma\gamma=\lambda_t}}\left(1-\frac1q\right)^{\kappa((n_\gamma))},$$
where $\kappa((n_\gamma))=\#\{\gamma:n_\gamma>0\}$. By \cite[Lemma 3.13]{CCL26}, we have $\delta^{-1/2}(t)=q^{\sum_{i=1}^m r_i}$. Hence
$$|X(\dot w_0 t)|\le C\,q^{\sum_{i=1}^m r_i}\,R(\lambda_t),$$
where $R(\lambda_t)$ is the Kostant partition function. By \cite[Lemma 3.14]{CCL26},
$$R(\lambda_t)\le \left(1+\sum_{i=1}^m r_i\right)^{|\Phi^{\vee,+}|}.$$
Combining these proves the proposition.
\end{proof}

We now recall Stevens' trick estimate from \cite[Corollary 5.4]{CCL26}. Although stated for the standard generic character in \cite{CCL26}, the same estimate holds for any fixed generic character $\psi$.
For $d>0$ sufficiently large, there exists an integer $\ell>0$, depending on $d$ and $t$, such that
$$|\mathrm{Kl}_d(t)|\le q^{m(d+k/2)}\sum_{y\in[-\ell,d]^m}|X_d(\dot w_0 t)_y|\,q^{\frac12\sum_{i=1}^m y_i},$$
where $k$ is a fixed nonnegative integer depending on $F$.

\begin{defn}[Fibre estimate]\label{def:fibre-estimate}
We say that the \emph{fibre estimate} holds for a given $d$ if there exist
constants $\varepsilon>0$ and $C>0$, possibly depending on $d$, such that for
every $t=t_0\prod_{i=1}^m\alpha_i^\vee(\varpi^{r_i})\in A$ with $t_0\in A(\mathcal O)$ and $r_i\ge0$, every
$\ell>d$ chosen as in Lemma~\ref{lem:ell-growth} for this $t$ and $d$, and
every $y\in[-\ell,d]^m$, one has
$$|X_d(\dot w_0 t)_y|\,q^{\frac12\sum_{i=1}^m y_i}
\le C\,q^{(1-\varepsilon)\sum_{i=1}^m r_i}.$$
\end{defn}

\begin{thm}\label{thm:fibre-implies-kl}
If the fibre estimate holds for all sufficiently large $d$, then $\{\mathrm{Kl}_d\}_{d\ge 1}$ satisfies a nontrivial bound. 
More precisely, there exist constants $\varepsilon_d>0$, $C_d>0$ such that
$$|\mathrm{Kl}_d(t)\delta^{1/2-\varepsilon_d}(t)|\le C_d,\qquad t\in A.$$
\end{thm}

\begin{proof}
By Stevens' trick estimate,
$$|\mathrm{Kl}_d(t)|\le q^{m(d+k/2)}\sum_{y\in[-\ell,d]^m}|X_d(\dot w_0 t)_y|\,q^{\frac12\sum_{i=1}^m y_i}.$$
Applying the assumed fibre bound to each summand gives
$$|\mathrm{Kl}_d(t)|\le C\,q^{m(d+k/2)}\sum_{y\in[-\ell,d]^m}q^{(1-\varepsilon)\sum_{i=1}^m r_i}.$$
By Lemma~\ref{lem:ell-growth}, $(\ell+d+1)^m\le C_1\left(1+\sum_{i=1}^m r_i\right)^m$ for some constant $C_1>0$. Hence
$$|\mathrm{Kl}_d(t)|\le C_2\left(1+\sum_{i=1}^m r_i\right)^m q^{(1-\varepsilon)\sum_{i=1}^m r_i}.$$
Since $\delta(t)=q^{-2\sum_i r_i}$, we have $\delta^{1/2-\varepsilon_d}(t)=q^{-(1-2\varepsilon_d)\sum_i r_i}$. Therefore
$$|\mathrm{Kl}_d(t)\delta^{1/2-\varepsilon_d}(t)|\le C_2\left(1+\sum_{i=1}^m r_i\right)^m q^{(-\varepsilon+2\varepsilon_d)\sum_i r_i}.$$
Choosing $\varepsilon_d=\varepsilon/4$, we may absorb the polynomial factor by the exponential decay, and the result follows.
\end{proof}

\subsection{Reduction to the shallow fibre estimate}\label{sect:reduce_to_shallow}

Our next goal is to reduce the fibre estimate to the case $d=0$ and then to a shallow range of the valuation vector $y$.

\begin{prop}\label{prop:reduct-d0}
If the fibre estimate holds for $d=0$, then it holds for every sufficiently large $d$. 
\end{prop}

\begin{proof}
Consider the commutative diagram
$$
\begin{tikzcd}
N_d\backslash (N\dot w_0tN\cap K_d)/N_d
  \arrow[d, "i", hook]
  \arrow[rr, "u'", hook]
& &
N/N_d
  \arrow[dd, "j", two heads]
  \arrow[rr, "\psi:=\nu\circ\theta'"]
& &
{[-\ell,d]^m}
  \arrow[dd, "\iota", two heads]
\\
N_d\backslash (N\dot w_0tN\cap K_0)/N_d
  \arrow[d, "i'", two heads]
& &
& &
\\
N_0\backslash (N\dot w_0tN\cap K_0)/N_0
  \arrow[rr, "u'", hook]
& &
N/N_0
  \arrow[rr, "\psi':=\nu'\circ\theta'"]
& &
{[-\ell,0]^m}
\end{tikzcd}
$$
Here $\iota$ is given by $\iota\bigl((x_i)_{1\le i\le m}\bigr)=\bigl(\min(x_i,0)\bigr)_{1\le i\le m}$, and we set
$$\nu(x)=\min\{v(x),d\},\qquad\nu'(x)=\min\{v(x),0\}.$$
Thus, for any $y\in[-\ell,d]^m$, we have
$$\begin{aligned}|X_d(\dot w_0t)_y|=|\psi^{-1}(y)|&\le |N_0/N_d|\cdot |j(\psi^{-1}(y))|\\&\le |N_0/N_d|\cdot |\psi'^{-1}(\iota(y))|\\&\le |N_0/N_d|\cdot |X_0(\dot w_0t)_{\iota(y)}|.\end{aligned}$$
By the assumed estimate for the base case $d=0$, we have
$$|X_0(\dot w_0t)_y|\,q^{\frac12\sum_{i=1}^m y_i}\le C\,q^{(1-\varepsilon)\sum_{i=1}^m r_i},\qquad y\in[-\ell,0]^m.$$
Applying this bound to the inequality above yields
$$\begin{aligned}|X_d(\dot w_0t)_y|\,q^{\frac12\sum_{i=1}^m y_i}&\le C_d\cdot |X_0(\dot w_0t)_{\iota(y)}|q^{\frac12\sum_{i=1}^m (\min(y_i,0)+d)}\\
&\le C_d'\cdot q^{(1-\varepsilon)\sum_{i=1}^m r_i}.\end{aligned}$$
This proves the fibre estimate for $d$.
\end{proof}

\begin{defn}[Shallow fibre estimate]\label{def:shallow-fibre}
We say that the \emph{shallow fibre estimate} holds if there exist constants
$\varepsilon>0$, $\varepsilon'>0$ and $C>0$ such that for every
$t=t_0\prod_{i=1}^m\alpha_i^\vee(\varpi^{r_i})\in A$ with $t_0\in A(\mathcal O)$ and $r_i\ge0$, every
$\ell>0$ chosen as in Lemma~\ref{lem:ell-growth} for this $t$ (with $d=0$),
and every $y\in[-\ell,0]^m$ satisfying
$\sum_{i=1}^m y_i\ge-\varepsilon'\sum_{i=1}^m r_i$, one has
$$|X_0(\dot w_0t)_y|\,q^{\frac12\sum_{i=1}^m y_i}
\le C\,q^{(1-\varepsilon)\sum_{i=1}^m r_i}.$$
\end{defn}

\begin{prop}\label{prop:shallow-implies-fibre}
The shallow fibre estimate implies the fibre estimate for $d=0$. Consequently, by Proposition~\ref{prop:reduct-d0} and Theorem~\ref{thm:fibre-implies-kl}, it suffices to prove the shallow fibre estimate in order to establish the nontrivial bound for $\mathrm{Kl}_d(t)$.
\end{prop}

\begin{proof}
Assume that the shallow fibre estimate holds with constants $\varepsilon,\varepsilon'>0$ and $C>0$. Let $y\in[-\ell,0]^m$. We distinguish two cases.

If $\sum_{i=1}^m y_i\ge -\varepsilon'\sum_{i=1}^m r_i$, then the desired bound is exactly the shallow fibre estimate.

Otherwise, $\sum_{i=1}^m y_i<-\varepsilon'\sum_{i=1}^m r_i$. By Proposition~\ref{prop:trivial-Xd},
$$|X_0(\dot w_0t)_y|\le \#X_0(\dot w_0t)\le C_0\,q^{\sum_{i=1}^m r_i}\left(1+\sum_{i=1}^m r_i\right)^{|\Phi^{\vee,+}|}.$$
Hence
$$|X_0(\dot w_0t)_y|\,q^{\frac12\sum_{i=1}^m y_i}\le C_0\,q^{(1-\frac{\varepsilon'}2)\sum_{i=1}^m r_i}\left(1+\sum_{i=1}^m r_i\right)^{|\Phi^{\vee,+}|}.$$
Replacing $\varepsilon$ by $\min\{\varepsilon,\varepsilon'/4\}$ if necessary, we may absorb the polynomial factor by $q^{\varepsilon\sum_{i=1}^m r_i}$, and obtain
$$|X_0(\dot w_0t)_y|\,q^{\frac12\sum_{i=1}^m y_i}\le C_1\,q^{(1-\varepsilon)\sum_{i=1}^m r_i}$$
for some constant $C_1>0$ independent of $t$ and $y$. This proves the fibre estimate for $d=0$.
\end{proof}

\subsection{Perturbation method}\label{sect:perturbation}

Throughout this subsection we take $d=0$ and write $X(\dot w_0t)_y:=X_0(\dot w_0t)_y$. Let $\nu(x)=\min\{v(x),0\}$ for $x\in F$.

For $y=(y_1,\ldots,y_m)\in [-\ell,0]^m$ and $r=(r_1,\ldots,r_m)\in \mathbb Z_{\ge0}^m$, define the set
$$\mathcal T_{y,r}:=\left\{\mathbf x=(x_\gamma)_\gamma\in F^{\Phi^+}\ \Big|\ \nu(x_{\alpha_i})=y_i,\quad v(p_{ij}(\mathbf x))\ge -r_i\text{ for all } i,j\right\}.$$
We equip this set with the equivalence relation
$$\mathbf x\sim \mathbf x'\iff n(\mathbf x)N_0 = n(\mathbf x')N_0.$$

In this subsection we will prove the following theorem, which reduces the shallow fibre estimate to a combinatorial problem.

\begin{thm}[Perturbation method]\label{thm:perturbation}
Suppose that there exist constants $\varepsilon>0$ and $C>0$ such that, 
for all $r=(r_1,\ldots,r_m)\in\mathbb Z_{\ge0}^m$,
$$|\mathcal T_{\mathbf 0,r}/{\sim}|\le C\,q^{(1-\varepsilon)\sum_{i=1}^m r_i}.$$
Then the Kloosterman sums $\{\mathrm{Kl}_d\}_{d\ge1}$ satisfy a nontrivial bound.
\end{thm}

The proof is in two steps. Lemma~\ref{lem:X-fibre-bound} bounds each fibre
$X(\dot w_0t)_y$ by $|\mathcal T_{y,r}/{\sim}|$, while Lemma~\ref{lem:Tyr-to-T0r}
compares the latter one with $\mathcal T_{\mathbf 0,r}/{\sim}$; combining
the two yields the shallow fibre estimate, and the theorem follows from
Proposition~\ref{prop:shallow-implies-fibre}.

\begin{lem}\label{lem:X-fibre-bound}
For every $y\in[-\ell,0]^m$ and every $r\in\mathbb Z_{\ge0}^m$, one has
$$|X(\dot w_0t)_y|\le |\mathcal T_{y,r}/{\sim}|,$$
where $t=t_0\prod_{i=1}^m\alpha_i^\vee(\varpi^{r_i})$ with arbitrary $t_0\in A(\mathcal O)$.
\end{lem}

\begin{proof}
Let $[g]\in X(\dot w_0t)_y$ and choose a representative $g\in K_0$. 
Write $g=n_1\dot w_0tn_2$ with $n_1,n_2\in N$, and choose $\mathbf x\in F^{\Phi^+}$ 
with $n(\mathbf x)N_0=n_2N_0$, i.e. $n_2=n(\mathbf x)n_0$ for some $n_0\in N_0$. We have
$$n_1\dot w_0t\,n(\mathbf x)=gn_0^{-1}\in K_0 .$$
Using that $\ell_i$ is fixed by $N$, and that $\dot w_0t\,n(\mathbf x)=n_1^{-1}gn_0^{-1}$, 
the left $N$-invariance gives
$$p_{ij,t}(\mathbf x)=\langle\ell_i,\pi_{\omega_i}(\dot w_0t\,n(\mathbf x))e_{ij}\rangle=\langle\ell_i,\pi_{\omega_i}(gn_0^{-1})e_{ij}\rangle ,$$
so the matrix coefficient is evaluated at the point $gn_0^{-1}\in K_0$. 
After multiplying each $\ell_i$ by a suitable nonzero scalar, we may
assume that
$$
\left\langle \ell_i,\pi_{\omega_i}(k)e_{ij}\right\rangle
\in\mathcal O
\qquad
\text{for every }k\in K_0\text{ and every }j.
$$
This is possible because $K_0$ is compact. Then $v(p_{ij,t}(\mathbf x))\ge0$ for all $i,j$.
By Corollary~\ref{cor:plucker-torus}, $p_{ij,t}(\mathbf x)=\omega_i(t)p_{ij}(\mathbf x)$. Since $t=t_0\prod_{k=1}^m\alpha_k^\vee(\varpi^{r_k})$ and $\omega_i(t_0)\in\mathcal O^\times$, we have $v(\omega_i(t))=r_i$. Hence
$$v(p_{ij}(\mathbf x))\ge-r_i\qquad(1\le i\le m,\ j).$$
Moreover, since the $\alpha_i$-coordinate of $n_2=n(\mathbf x)n_0$ is $x_{\alpha_i}+(n_0)_{\alpha_i}$ 
and $\nu((n_2)_{\alpha_i})=y_i$, we have $\nu(x_{\alpha_i})=y_i$.
Hence $\mathbf x\in\mathcal T_{y,r}$. 

The map $X(\dot w_0t)_y\to\mathcal T_{y,r}/{\sim}$, $[g]\mapsto[\mathbf x]$, is well defined, 
because another representative of $[g]$ replaces $n_2$ by $n_2m_0$ with $m_0\in N_0$, 
hence $\mathbf x$ by a point of the same class; and it is injective, 
because $\mathbf x\sim\mathbf x'$ implies
$$u'([g])=n(\mathbf x)N_0=n(\mathbf x')N_0=u'([g']),$$
so that $[g]=[g']$ by the injectivity of $u'$. This proves the lemma.
\end{proof}

\begin{lem}[A perturbation map] \label{lem:Tyr-to-T0r}
Let $y=(y_1,\ldots,y_m)\in[-\ell,0]^m$ and $r\in\mathbb Z_{\ge0}^m$. Put
$$u_\gamma:=\varpi^{-\sum_{i=1}^m\kappa_{\alpha_i}(\gamma)y_i}\qquad(\gamma\in\Phi^+),$$
and consider the following \emph{perturbation map}
$$\phi:F^{\Phi^+}\longrightarrow F^{\Phi^+},\qquad \phi(\mathbf x)_\gamma:=u_\gamma x_\gamma .$$
Then $\phi(\mathcal T_{y,r})\subseteq\mathcal T_{\mathbf 0,r}$, and $\phi$ is compatible with the equivalence relations. Hence $\phi$ induces a map
$$\tilde\phi:\ \mathcal T_{y,r}/{\sim}\ \longrightarrow\ \mathcal T_{\mathbf 0,r}/{\sim}.$$
Moreover, every fibre of $\tilde\phi$ has at most $q^{-\sum_{i=1}^m\rho(\alpha_i)y_i}$ elements, where $\rho(\alpha):=\sum_{\gamma\in\Phi^+}\kappa_\alpha(\gamma)$ depends only on the root system.
\end{lem}

\begin{proof}
By Corollary~\ref{cor:plucker-torus}, $p_{ij}$ is $\omega_i-\chi_{ij}$-homogeneous. Replacing $\mathbf x$ by $\phi(\mathbf x)$ therefore multiplies each monomial $\mathbf x^{\mathbf k}$ of $p_{ij}$ by
$$\prod_{\gamma\in\Phi^+}u_\gamma^{k_\gamma}
=\varpi^{-\sum_{k=1}^m\left(\sum_{\gamma\in\Phi^+}\kappa_{\alpha_k}(\gamma)k_\gamma\right)y_k}
=\varpi^{-\sum_{k=1}^m\kappa_{\alpha_k}(\omega_i-\chi_{ij})y_k};$$
that is,
$$p_{ij}(\phi(\mathbf x))=\varpi^{-\sum_{k=1}^m\kappa_{\alpha_k}(\omega_i-\chi_{ij})y_k}\,p_{ij}(\mathbf x)
\qquad\text{for all }i,j .$$
Since $y_k\le0$ and $\kappa_{\alpha_k}(\omega_i-\chi_{ij})\ge0$, the scalar in front lies in $\mathcal O$, whence $v(p_{ij}(\phi(\mathbf x)))\ge v(p_{ij}(\mathbf x))\ge-r_i$. Moreover $u_{\alpha_i}=\varpi^{-y_i}$, so $v(\phi(\mathbf x)_{\alpha_i})=v(x_{\alpha_i})-y_i$; this equals $0$ when $y_i<0$, because then $v(x_{\alpha_i})=y_i$, and it is $\ge0$ when $y_i=0$. Hence $\nu(\phi(\mathbf x)_{\alpha_i})=0$ in both cases, and $\phi(\mathbf x)\in\mathcal T_{\mathbf 0,r}$.

For the compatibility with the equivalence relations, consider
$$\sigma:N\longrightarrow N,\qquad
\sigma\Bigl(\prod_{\gamma\in\Phi^+}X_\gamma(t_\gamma)\Bigr)
:=\prod_{\gamma\in\Phi^+}X_\gamma(u_\gamma t_\gamma),$$
so that $n(\phi(\mathbf x))=\sigma(n(\mathbf x))$. Since $\kappa_{\alpha_i}(\gamma+\delta)=\kappa_{\alpha_i}(\gamma)+\kappa_{\alpha_i}(\delta)$, the family $(u_\gamma)_\gamma$ is multiplicative, $u_{\gamma+\delta}=u_\gamma u_\delta$; together with the Chevalley commutator formula \cite[Theorem 5.2.2]{Car72} this shows that $\sigma$ is a group automorphism of $N$, with inverse given by $(u_\gamma^{-1})_\gamma$. Moreover $v(u_\gamma)=-\sum_{i=1}^m\kappa_{\alpha_i}(\gamma)y_i\ge0$, so $u_\gamma\mathcal O\subseteq\mathcal O$ and $\sigma(N_0)\subseteq N_0$. If $\mathbf x\sim\mathbf x'$, that is, $n(\mathbf x')=n(\mathbf x)n_0$ with $n_0\in N_0$, then
$$n(\phi(\mathbf x'))=\sigma\bigl(n(\mathbf x')\bigr)
=\sigma\bigl(n(\mathbf x)\bigr)\sigma(n_0)
=n(\phi(\mathbf x))\cdot\sigma(n_0),$$
and $\sigma(n_0)\in N_0$; therefore $\phi(\mathbf x)\sim\phi(\mathbf x')$, so that $\tilde\phi$ is well defined.

It remains to bound the fibres of $\tilde\phi$. Let $C_0=mN_0$ be a coset of $N_0$ in $N$. Then
$$\phi^{-1}(C_0)=\{\mathbf x:\ n(\mathbf x)\in\sigma^{-1}(mN_0)\}
=\sigma^{-1}(m)\cdot\sigma^{-1}(N_0),$$
and
$$\sigma^{-1}(N_0)=\prod_{\gamma\in\Phi^+}X_\gamma\bigl(u_\gamma^{-1}\mathcal O\bigr)
=\prod_{\gamma\in\Phi^+}X_\gamma\bigl(\mathfrak p^{-v(u_\gamma)}\bigr)$$
is a subgroup of $N$ containing $N_0$, of index
$$[\,\sigma^{-1}(N_0):N_0\,]=\prod_{\gamma\in\Phi^+}q^{v(u_\gamma)} .$$
Hence $\phi^{-1}(C_0)$ is a union of exactly that many cosets of $N_0$, and these are precisely the classes of $\sim$ contained in $\phi^{-1}(C_0)$. Consequently every fibre of $\tilde\phi$ has at most $\prod_{\gamma\in\Phi^+}q^{v(u_\gamma)}$ elements. Finally,
$$\sum_{\gamma\in\Phi^+}v(u_\gamma)
=-\sum_{\gamma\in\Phi^+}\sum_{i=1}^m\kappa_{\alpha_i}(\gamma)y_i
=-\sum_{i=1}^m\rho(\alpha_i)y_i,$$
which gives the asserted bound $q^{-\sum_{i=1}^m\rho(\alpha_i)y_i}$.
\end{proof}

We can now prove Theorem~\ref{thm:perturbation}.

\begin{proof}[Proof of Theorem~\ref{thm:perturbation}]
Let $\varepsilon_0>0$ and $C_0>0$ be as in the hypothesis, and put $\rho_{\max}:=\max_{\alpha\in\Delta}\rho(\alpha)$. Choose
$$\varepsilon':=\frac{\varepsilon_0}{2\rho_{\max}},\qquad \varepsilon:=\frac{\varepsilon_0}{2}.$$
Let $r\in\mathbb Z_{\ge0}^m$, let $t_0\in A(\mathcal O)$, put $t=t_0\prod_{i=1}^m\alpha_i^\vee(\varpi^{r_i})$, and let $\ell>0$ be chosen as in Lemma~\ref{lem:ell-growth}. Let $y\in[-\ell,0]^m$ satisfy
$$\sum_{i=1}^m y_i\ \ge\ -\varepsilon'\sum_{i=1}^m r_i .$$
Choose $z=(z_\alpha)_{\alpha\in\Delta}$ with $v(z_{\alpha_i})=y_i$, so that Lemma~\ref{lem:Tyr-to-T0r} applies. By Lemma~\ref{lem:X-fibre-bound} and Lemma~\ref{lem:Tyr-to-T0r},
$$|X(\dot w_0t)_y|\ \le\ |\mathcal T_{y,r}/{\sim}|\ \le\ q^{-\sum_{\alpha\in\Delta}\rho(\alpha) y_\alpha}\,|\mathcal T_{\mathbf 0,r}/{\sim}| .$$
Since $y_\alpha\le0$ and $\rho(\alpha)\le\rho_{\max}$ for all $\alpha$, we have
$$-\sum_{\alpha\in\Delta}\rho(\alpha) y_\alpha
\ \le\ \rho_{\max}\sum_{\alpha\in\Delta}(-y_\alpha)
\ =\ \rho_{\max}\Bigl(-\sum_{i=1}^m y_i\Bigr)
\ \le\ \rho_{\max}\varepsilon'\sum_{i=1}^m r_i ,$$
while the hypothesis gives $|\mathcal T_{\mathbf 0,r}/{\sim}|\le C_0\,q^{(1-\varepsilon_0)\sum_{i=1}^m r_i}$. Hence
$$|X(\dot w_0t)_y|\ \le\ C_0\,q^{(1-\varepsilon_0+\rho_{\max}\varepsilon')\sum_{i=1}^m r_i}
\ =\ C_0\,q^{(1-\varepsilon)\sum_{i=1}^m r_i}.$$
As $y_i\le0$, we have
$$|X(\dot w_0t)_y|\,q^{\frac12\sum_{i=1}^m y_i}\ \le\ C_0\,q^{(1-\varepsilon)\sum_{i=1}^m r_i}.$$
The constants $\varepsilon,\varepsilon',C_0$ depend neither on $t$ nor on $\ell$ nor on $y$. Thus the shallow fibre estimate holds, and the theorem follows from Proposition~\ref{prop:shallow-implies-fibre}.
\end{proof}

To make this theorem practically useful, we introduce the following corollary, which will be used frequently in the next section.

Fix a section $\mathcal R\subset F$ of $F/\mathcal O$ containing $0$, and let $q_\gamma\in F[\mathbf x]$ $(\gamma\in\Phi^+)$ be polynomials such that
$$q_{\alpha_i}=x_{\alpha_i},\qquad q_\gamma=x_\gamma+f_\gamma,\qquad f_\gamma\in F[x_\beta:\beta<\gamma].$$
For $r=(r_1,\ldots,r_m)\in\mathbb Z_{\ge0}^m$, define \emph{the Pl\"ucker-constrained set}
$$S_r:=\left\{\mathbf x=(x_\gamma)_\gamma\in F^{\Phi^+}\ \Big|\ x_{\alpha_i}=0,\quad q_\gamma(\mathbf x)\in\mathcal R\text{ for all }\gamma\in\Phi^+\setminus\Delta,\quad v(p_{ij}(\mathbf x))\ge -r_i\text{ for all }i,j\right\}.$$
In the next section, the polynomials $q_\gamma$ allow us to change coordinates before counting, and the section $\mathcal R$ allows us to prescribe the values taken by these coordinates; for instance, $q_\gamma=x_\gamma$ together with $\mathcal R\ni0$ fixes the simple-root coordinates at $0$.

\begin{cor}\label{cor:practical}
Suppose that there exist constants $\varepsilon>0$ and $C>0$ such that,
for all $r\in\mathbb Z_{\ge0}^m$,
$$|S_r|\le C\,q^{(1-\varepsilon)\sum_{i=1}^m r_i}.$$
Then the Kloosterman sums $\{\mathrm{Kl}_d\}_{d\ge1}$ satisfy a nontrivial bound.
\end{cor}

The proof of the corollary rests on the two lemmas below: the first provides a canonical representative of each class of $\sim$, and the second shows that such a representative satisfies the Pl\"ucker conditions up to a constant shift of the exponents.

Write $\mathbf x n_0$ for the unique element of $F^{\Phi^+}$ with
$$n(\mathbf x n_0)=n(\mathbf x)n_0\qquad(\mathbf x\in F^{\Phi^+},\ n_0\in N_0);$$
thus $\mathbf x\mapsto\mathbf x n_0$ is the right action of $N_0$ on $F^{\Phi^+}$ obtained from the bijection $\mathbf x\leftrightarrow n(\mathbf x)$.

\begin{lem}\label{lem:normal-form}
With $\mathcal R$ and $q_\gamma$ as above, put
$$B_{\mathcal R}:=\left\{\mathbf x=(x_\gamma)_\gamma\in F^{\Phi^+}\ \Big|\ q_\gamma(\mathbf x)\in\mathcal R\text{ for all }\gamma\in\Phi^+\right\}.$$
Then the map
$$B_{\mathcal R}\longrightarrow N/N_0,\qquad \mathbf x\longmapsto n(\mathbf x)N_0,$$
is a bijection. Equivalently, every equivalence class of $\sim$ in $F^{\Phi^+}$ meets $B_{\mathcal R}$ in exactly one point.
\end{lem}

\begin{proof}
We show that for every $\mathbf x\in F^{\Phi^+}$ there is exactly one $n_0\in N_0$ with $\mathbf x n_0\in B_{\mathcal R}$. Fix $\mathbf x$ and determine $n_0=\prod_{\gamma\in\Phi^+}X_\gamma(t_\gamma)$ by choosing the coordinates $t_\gamma$ one height at a time, in increasing height; we shall see that each choice is forced.

Two elementary facts are used. First, for $\gamma\in\Phi^+$ and $t\in F$, the point $\mathbf x X_\gamma(t)$ has $\gamma$-coordinate $x_\gamma+t$ and agrees with $\mathbf x$ in every other coordinate of height $\le\mathrm{ht}(\gamma)$.
Second, consequently, if $\mathbf y=\mathbf x n_0^-$ with $n_0^-=\prod_{\mathrm{ht}(\delta)<\mathrm{ht}(\gamma)}X_\delta(t_\delta)$, then
$$q_\gamma\bigl(\mathbf y X_\gamma(t)\bigr)=q_\gamma(\mathbf y)+t ,$$
because the coordinates occurring in $f_\gamma$ have height $<\mathrm{ht}(\gamma)$ and are therefore unchanged, while the $\gamma$-coordinate gains $t$.

Now suppose that $t_\delta$ has been chosen for all $\delta$ of height $<\mathrm{ht}(\gamma)$, and write $\mathbf y=\mathbf x n_0^-$ as above. By the second fact, the requirement $q_\gamma(\mathbf y X_\gamma(t))\in\mathcal R$ reads
$$q_\gamma(\mathbf y)+t\in\mathcal R ;$$
as $t$ ranges over $\mathcal O$, the left-hand side ranges over the coset $q_\gamma(\mathbf y)+\mathcal O$, which the section $\mathcal R$ meets in exactly one point. Hence this requirement determines a unique $t\in\mathcal O$, which we take as $t_\gamma$; by the first fact, this choice leaves unchanged the conditions already imposed at smaller heights.

Inducting over the height we obtain $n_0\in N_0$ with $q_\gamma(\mathbf x n_0)\in\mathcal R$ for every $\gamma$, that is, $\mathbf x n_0\in B_{\mathcal R}$. For uniqueness, let $n_0'$ be a second such element and let $\gamma$ be of minimal height with $t'_\gamma\neq t_\gamma$. Then $\mathbf x n_0$ and $\mathbf x n_0'$ agree in every coordinate of height $<\mathrm{ht}(\gamma)$, while their $\gamma$-coordinates differ by $t'_\gamma-t_\gamma\neq0$; since $f_\gamma$ involves only coordinates of smaller height, the two numbers $q_\gamma(\mathbf x n_0)$ and $q_\gamma(\mathbf x n_0')$ both lie in $\mathcal R$ and differ by $t'_\gamma-t_\gamma\in\mathcal O$, contradicting that $\mathcal R$ meets each coset of $\mathcal O$ exactly once. Hence $n_0$ is unique, and $B_{\mathcal R}\to N/N_0$ is bijective.
\end{proof}

\begin{lem}\label{lem:shift}
Fix a weight basis $\{e_{ij}\}$ of $V_{\omega_i}$ as in the definition of $p_{ij}$, and for $n_0\in N_0$ write $\pi_{\omega_i}(n_0)e_{ij}=\sum_k a_{ijk}(n_0)e_{ik}$. Put
$$C_0:=\max_{i,j,k}\Bigl(0,\,-\min_{n_0\in N_0}v\bigl(a_{ijk}(n_0)\bigr)\Bigr)<\infty .$$
If $y\in[-\ell,0]^m$, $r\in\mathbb Z_{\ge0}^m$, $\mathbf x\in\mathcal T_{y,r}$ and $\mathbf x\sim\mathbf x'$, then $\mathbf x'\in\mathcal T_{y,\,r+C_0\mathbf 1}$.
\end{lem}

\begin{proof}
Since $\mathbf x\sim\mathbf x'$, there is $n_0\in N_0$ with $n(\mathbf x')=n(\mathbf x)n_0$.

We first check that $C_0$ is finite. For each $i,j,k$, the function
$$N_0\longrightarrow F,\qquad n_0\longmapsto a_{ijk}(n_0)$$
is continuous. Since $N_0$ is compact, its image is bounded in $F$.
There are only finitely many triples $(i,j,k)$, so there exists an
integer $M\ge0$ such that
$$v(a_{ijk}(n_0))\ge-M\qquad\text{for every }n_0\in N_0\text{ and every }i,j,k.$$
Thus $C_0<\infty$.

Now fix $i,j$. Using $n(\mathbf x')=n(\mathbf x)n_0$ and $\pi_{\omega_i}(n_0)e_{ij}=\sum_ka_{ijk}(n_0)e_{ik}$, we obtain
$$p_{ij}(\mathbf x')
=\langle\ell_i,\pi_{\omega_i}(\dot w_0n(\mathbf x)n_0)e_{ij}\rangle
=\sum_k a_{ijk}(n_0)\,\langle\ell_i,\pi_{\omega_i}(\dot w_0n(\mathbf x))e_{ik}\rangle
=\sum_k a_{ijk}(n_0)\,p_{ik}(\mathbf x).$$
Since $\mathbf x\in\mathcal T_{y,r}$ we have $v(p_{ik}(\mathbf x))\ge-r_i$ for every $k$, while $v(a_{ijk}(n_0))\ge-C_0$; therefore
$$v\bigl(p_{ij}(\mathbf x')\bigr)\ \ge\ -r_i-C_0 .$$

As in the proof of Lemma~\ref{lem:X-fibre-bound}, the $\alpha_i$-coordinate of $n(\mathbf x)n_0$ equals $x_{\alpha_i}+(n_0)_{\alpha_i}$ with $(n_0)_{\alpha_i}\in\mathcal O$; hence $\nu(x'_{\alpha_i})=\nu(x_{\alpha_i})=y_i$. Thus $\mathbf x'$ satisfies $\nu(x'_{\alpha_i})=y_i$ and $v(p_{ij}(\mathbf x'))\ge-r_i-C_0$ for all $i,j$, that is, $\mathbf x'\in\mathcal T_{y,\,r+C_0\mathbf 1}$.
\end{proof}

\begin{proof}[Proof of Corollary~\ref{cor:practical}]
Let $C_0$ be the constant of Lemma~\ref{lem:shift}, and let $\mathbf x\in\mathcal T_{\mathbf 0,r}$. By Lemma~\ref{lem:normal-form} the map
$$B_{\mathcal R}\longrightarrow N/N_0,\qquad \mathbf y\longmapsto n(\mathbf y)N_0 ,$$
is a bijection, so there is a unique $\mathbf x_0\in B_{\mathcal R}$ with $n(\mathbf x_0)N_0=n(\mathbf x)N_0$, that is, $\mathbf x_0\sim\mathbf x$. Since $\mathbf x\in\mathcal T_{\mathbf 0,r}$, Lemma~\ref{lem:shift} gives
$$\mathbf x_0\in\mathcal T_{\mathbf 0,\,r+C_0\mathbf 1},$$
that is, $\nu((\mathbf x_0)_{\alpha_i})=0$ and $v(p_{ij}(\mathbf x_0))\ge-r_i-C_0$ for all $i,j$. From $\nu((\mathbf x_0)_{\alpha_i})=0$ we get $(\mathbf x_0)_{\alpha_i}\in\mathcal O$. Since $(\mathbf x_0)_{\alpha_i}=q_{\alpha_i}(\mathbf x_0)\in\mathcal R$ and $\mathcal R$ contains $0$, this forces $(\mathbf x_0)_{\alpha_i}=0$.
Hence $\mathbf x_0\in S_{r+C_0\mathbf 1}$. The point $\mathbf x_0$ depends only on the class of $\mathbf x$, and distinct classes have distinct points in $B_{\mathcal R}$; therefore
$$|\mathcal T_{\mathbf 0,r}/{\sim}|\le|S_{r+C_0\mathbf 1}| .$$
Applying the hypothesis to the parameter $r+C_0\mathbf 1\in\mathbb Z_{\ge0}^m$ gives
$$|S_{r+C_0\mathbf 1}|\le C\,q^{(1-\varepsilon)\sum_{i=1}^m(r_i+C_{0})}
=\Bigl(C\,q^{(1-\varepsilon)mC_{0}}\Bigr)q^{(1-\varepsilon)\sum_{i=1}^mr_i},$$
so that $|\mathcal T_{\mathbf 0,r}/{\sim}|\le C_1\,q^{(1-\varepsilon)\sum_{i=1}^mr_i}$ for all $r\in\mathbb Z_{\ge0}^m$, with $C_1:=C\,q^{(1-\varepsilon)mC_0}$. The assertion now follows from Theorem~\ref{thm:perturbation}.
\end{proof}

\subsection{Finiteness of Pl\"ucker-constrained sets}\label{sect:finite}

We prove the finiteness of $S_r$ and Lemma~\ref{lem:ell-growth}.
Both follow from a bound on the root coordinates. Put
$$\CP_r:=\bigl\{\mathbf x\in F^{\Phi^+}:\ v(p_{ij}(\mathbf x))\ge -r_i\
\text{for all }i,j\bigr\}.$$
Then the projection $\operatorname{proj}_\gamma:\CP_r\to F$,
$\mathbf x\mapsto x_\gamma$, has image contained in $\fp^{-k_\gamma}$ with
$k_\gamma$ bounded linearly in the $r_i$. 

\begin{prop}\label{prop:arale}
For every choice of the $\ell_i$ and of $A$-weight bases $\{e_{ij}\}_j$ of the
$V_{\omega_i}$ there is a constant $C\ge0$, depending only on the $\ell_i$ and the
$e_{ij}$, such that for every $r\in\mathbb Z_{\ge0}^m$ and every
$\gamma\in\Phi^+$,
$$\operatorname{proj}_\gamma(\CP_r)\subset\fp^{-k_\gamma},\qquad
k_\gamma:=\operatorname{ht}(\gamma)\Bigl(C+\max_{1\le i\le m}r_i\Bigr).$$
\end{prop}
\begin{proof}
Let $\mathbf x\in\CP_r$. By Lemma~\ref{lem:plucker-monomial}, for every
$\gamma\in\Phi^+$ there are indices $i=i(\gamma)$, $j=j(\gamma)$ such that
$$p_{i(\gamma),j(\gamma)}(\mathbf x)=c_\gamma x_\gamma+g_\gamma(\mathbf x),
\qquad c_\gamma\ne0,\qquad g_\gamma\in F[x_\beta:\beta<\gamma],$$
the polynomial $p_{i(\gamma),j(\gamma)}$ being $\gamma$-homogeneous. Write
$$g_\gamma=\sum_{\mathbf k}b_{\gamma,\mathbf k}\mathbf x^{\mathbf k};$$
since $g_\gamma$ is $\gamma$-homogeneous, every monomial satisfies
$\sum_{\beta\in\Phi^+}k_\beta\beta=\gamma$. Put
$$C_0:=\max\Bigl\{0,\ \max_{\gamma\in\Phi^+}v(c_\gamma),\
\max_{\gamma,\mathbf k}\bigl(-v(b_{\gamma,\mathbf k})\bigr)\Bigr\}.$$
We claim that, for every $\gamma\in\Phi^+$,
$$\operatorname{proj}_\gamma(\CP_r)\subset\fp^{-l_\gamma},\qquad
l_\gamma:=(3\operatorname{ht}(\gamma)-2)C_0+\operatorname{ht}(\gamma)r_\infty,
\qquad r_\infty:=\max_{1\le k\le m}r_k .$$

We prove the claim by induction on $\operatorname{ht}(\gamma)$. If
$\operatorname{ht}(\gamma)=1$, then $g_\gamma=0$, so the condition
$v(p_{i(\gamma),j(\gamma)}(\mathbf x))\ge -r_{i(\gamma)}$ reads
$v(c_\gamma x_\gamma)\ge -r_{i(\gamma)}$; hence
$v(x_\gamma)\ge -r_\infty-C_0$, and the claim holds.

Now let $\operatorname{ht}(\gamma)>1$ and assume the claim for every
$\beta\in\Phi^+$ with $\operatorname{ht}(\beta)<\operatorname{ht}(\gamma)$. Then
$$v(x_\gamma)\ \ge\ \min\{-r_\infty,\,v(g_\gamma(\mathbf x))\}-v(c_\gamma),$$
and
$$\begin{aligned}
v(g_\gamma(\mathbf x))
&\ge \min_{\mathbf k}\Bigl\{v(b_{\gamma,\mathbf k})
   +\sum_{\beta<\gamma}k_\beta\,v(x_\beta)\Bigr\}\\
&\ge \min_{\mathbf k}\Bigl\{-C_0-\sum_{\beta<\gamma}k_\beta
   \bigl((3\operatorname{ht}(\beta)-2)C_0+\operatorname{ht}(\beta)r_\infty\bigr)\Bigr\}\\
&\ge \min_{\mathbf k}\Bigl\{3C_0-\sum_{\beta<\gamma}k_\beta
   \operatorname{ht}(\beta)(3C_0+r_\infty)\Bigr\}\\
&\ge 3C_0-\operatorname{ht}(\gamma)(3C_0+r_\infty).
\end{aligned}$$
Here the third inequality uses $\sum_{\beta<\gamma}k_\beta\ge2$, which follows
from $\sum k_\beta\beta=\gamma$ together with $\beta<\gamma$, and the last one
uses $\sum_\beta k_\beta\operatorname{ht}(\beta)=\operatorname{ht}(\gamma)$.
Therefore
$$v(x_\gamma)\ \ge\ 2C_0-\operatorname{ht}(\gamma)(3C_0+r_\infty)\ \ge\ -l_\gamma ,$$
which proves the claim. Taking $C=3C_0$ yields
$\operatorname{proj}_\gamma(\CP_r)\subset\fp^{-k_\gamma}$ with
$k_\gamma=\operatorname{ht}(\gamma)(C+r_\infty)$, as asserted.
\end{proof}

\begin{cor}\label{cor:finiteness}
There are constants $C_0,C_1>0$, depending only on $G$ and the chosen Pl\"ucker
coordinates, such that for every $r\in\mathbb Z_{\ge0}^m$,
$$|S_r|\le C_0\,q^{C_1\max_{1\le i\le m}r_i}.$$
In particular, $S_r$ is finite.
\end{cor}
\begin{proof}
By Proposition~\ref{prop:arale} there is a constant $C>0$, depending only on the
chosen $\ell_i$ and $e_{ij}$, such that for every $\gamma\in\Phi^+$
$$\operatorname{proj}_\gamma(S_r)\subset\fp^{-k_\gamma},\qquad
k_\gamma=\operatorname{ht}(\gamma)\Bigl(C+\max_{1\le i\le m}r_i\Bigr),$$
since $S_r\subset\CP_r$. Since $x_{\alpha_i}=0$ for every simple root
$\alpha_i$, a point of $S_r$ is determined by its coordinates $x_\gamma$ with
$\gamma\in\Phi^+\setminus\Delta$. Fix such a $\gamma$ and suppose that the
coordinates $x_\beta$ with $\beta<\gamma$ are known; then $f_\gamma(\mathbf x)$ is
known, and $q_\gamma(\mathbf x)=x_\gamma+f_\gamma(\mathbf x)\in\CR$ means that
$x_\gamma$ lies in the set $\CR-f_\gamma(\mathbf x)$, which meets every coset of
$\CO$ in exactly one point. Together with
$x_\gamma\in\fp^{-k_\gamma}$,
this leaves at most $q^{k_\gamma}$ possibilities for $x_\gamma$. Inducting over
the height of $\gamma$ (in increasing order), we obtain
$$|S_r|\le\prod_{\gamma\in\Phi^+\setminus\Delta}q^{k_\gamma}
\le q^{|\Phi^+|H(C+r_\infty)},$$
where $H:=\max_{\gamma\in\Phi^+}\operatorname{ht}(\gamma)$ and
$r_\infty:=\max_{1\le i\le m}r_i$. Taking $C_1=|\Phi^+|H$ and
$C_0=q^{|\Phi^+|HC}$ gives the asserted bound.
\end{proof}

Finally, we give a proof of Lemma~\ref{lem:ell-growth}.
\begin{proof}[Proof of Lemma~\ref{lem:ell-growth}]
For every $x\in X_d(\dot w_0t)$ choose a representative $g\in K_d$ and write
$g=n_1\dot w_0t\,n_2$ with $n_1,n_2\in N$. Recall that the $\ell_i$ have been normalized so that
$$\left\langle\ell_i,\pi_{\omega_i}(k)e_{ij}\right\rangle
\in\mathcal O
\qquad
\text{for every }k\in K_0\text{ and every }i,j.$$
Since $g\in K_0$ and $\ell_i$ is fixed by $N$, we have, for all $i,j$,
$$p_{ij,t}(n_2)=\langle\ell_i,\pi_{\omega_i}(\dot w_0t\,n_2)e_{ij}\rangle
=\langle\ell_i,\pi_{\omega_i}(n_1^{-1}g)e_{ij}\rangle
=\langle\ell_i,\pi_{\omega_i}(g)e_{ij}\rangle\in\CO ,$$
so $v(p_{ij,t}(n_2))\ge0$. Since $p_{ij,t}=\omega_i(t)p_{ij}$ and
$v(\omega_i(t))=r_i$ (Corollary~\ref{cor:plucker-torus}), this gives
$v(p_{ij}(n_2))\ge -r_i$ for all $i,j$. Writing
$n_2=\prod_{\gamma\in\Phi^+}X_\gamma(z_\gamma)$, we conclude that
$(z_\gamma)_\gamma\in\CP_r$; hence, by Proposition~\ref{prop:arale},
$$v(z_\gamma)\ge -\operatorname{ht}(\gamma)\Bigl(C+\max_{1\le i\le m}r_i\Bigr)$$
for every $\gamma\in\Phi^+$. Enlarging $C$ if necessary, we may assume $C>d$. Set
$\ell:=C+\max_i r_i$. Then $\ell>d$ and
$v(z_\gamma)\ge-\ell\operatorname{ht}(\gamma)$ for all $\gamma$, so
$n_2\in N_{-\ell}$ by the definition of $N_{-\ell}$. Hence
$u'(x)\in N_{-\ell}/N_d$.

To bound $n_1$, apply the same argument to $g^{-1}$, using $\dot w_0^{-1}$ as the representative of the longest Weyl element. We have
$$g^{-1}=n_2^{-1}t^{-1}\dot w_0^{-1}n_1^{-1}=n_2^{-1}\dot w_0^{-1}t'n_1^{-1},\qquad t':=\dot w_0t^{-1}\dot w_0^{-1}\in A.$$
If $\lambda_t=\sum_{i=1}^m r_i\alpha_i^\vee$, then $w_0(\alpha_i^\vee)=-\alpha_{\sigma(i)}^\vee$ for a permutation $\sigma$ of the simple roots. Hence $\lambda_{t'}=w_0(-\lambda_t)$ has coefficients $r_i'=r_{\sigma^{-1}(i)}$, so that $\max_i r_i'=\max_i r_i$.

Applying the preceding coordinate-growth argument to $g^{-1}$, with $\dot w_0^{-1}$ in place of $\dot w_0$, gives $n_1^{-1}\in N_{-\ell}$ after enlarging the constant $C$ if necessary. Since $N_{-\ell}$ is a subgroup of $N$, it is stable under inversion, and hence $n_1\in N_{-\ell}$. Together with the estimate already proved for $n_2$, this yields $u(x)\in N_d\backslash N_{-\ell}$ and $u'(x)\in N_{-\ell}/N_d$.

Finally, taking $\ell=C+\max_i r_i$, we have $\ell>d$ and $\ell\le C(1+\sum_{i=1}^m r_i)$ after enlarging $C$ if necessary. This proves the lemma.
\end{proof}

\section{Lower rank cases}

We use the standard descriptions of the fundamental representations of the
classical and exceptional groups from \cite{FH91} throughout this section.

\subsection{\texorpdfstring{$\Sp_4$}{Sp4}}

Let $G=\Sp_4$, and let
$$J=\begin{pmatrix}
0 & 0 & 0 & 1 \\
0 & 0 & 1 & 0 \\
0 & -1 & 0 & 0 \\
-1 & 0 & 0 & 0
\end{pmatrix}.$$
Let $B=AN$ be the Borel subgroup of upper triangular matrices, with diagonal torus $A$ and unipotent radical $N$. Thus
$$A=\{\diag(a,b,b^{-1},a^{-1})\mid a,b\in F^\times\}.$$
Let $\Delta=\{\alpha,\beta\}$ be the set of simple roots, where $\alpha$ denotes the short root and $\beta$ the long root. The positive roots are
$$\Phi^+=\{\alpha,\beta,\alpha+\beta,2\alpha+\beta\}.$$
As characters of $A$, we have
$$\alpha(\diag(a,b,b^{-1},a^{-1}))=ab^{-1},\qquad \beta(\diag(a,b,b^{-1},a^{-1}))=b^2.$$
Hence the coroots are
$$\alpha^\vee(t)=\diag(t,t^{-1},t,t^{-1}),\qquad \beta^\vee(t)=\diag(1,t,t^{-1},1),$$
and the fundamental weights are
$$\omega_1(\diag(a,b,b^{-1},a^{-1}))=a,\qquad \omega_2(\diag(a,b,b^{-1},a^{-1}))=ab.$$

Let $V_1$ and $V_2$ be the fundamental representations with highest weights $\omega_1$ and $\omega_2$, respectively. Here $V_1$ is the $4$-dimensional standard representation, and $V_2$ is a $5$-dimensional representation of $\Sp_4$.

Let $e_1,e_2,e_3,e_4$ be the standard basis of $V_1$. Their weights are
$$\omega_1,\quad \omega_1-\alpha,\quad \alpha-\omega_1,\quad -\omega_1,$$
respectively. The highest weight vector of $V_1^*$ is $\ell_1:=e_4^\vee$. Hence we obtain four Pl\"ucker coordinates
$$x_j(g):=\langle g e_j,e_4^\vee\rangle=g_{4j},\qquad 1\le j\le 4.$$

For $V_2$, we use the decomposition
$$\bigwedge^2 V_1\cong V_2\oplus F\Omega,$$
where $\Omega=e_1\wedge e_4+e_2\wedge e_3$ is $G$-invariant. The highest weight vector of $V_2^*$ is $\ell_2:=e_3^\vee\wedge e_4^\vee$. Thus the matrix coefficients of $V_2$ are
$$y_{ij}(g):=\langle g\cdot(e_i\wedge e_j),\ell_2\rangle
=\det\begin{pmatrix}
g_{3i} & g_{3j} \\
g_{4i} & g_{4j}
\end{pmatrix}.$$
Since $\ell_2$ annihilates $\Omega$, we have
$$y_{14}+y_{23}=\langle g\cdot\Omega,\ell_2\rangle=0.$$
Therefore we obtain five linearly independent matrix coefficients from $V_2$, namely
$$\{y_{ij}(g):1\le i<j\le 4,\ (i,j)\neq(2,3)\}.$$

We now compute the polynomials $p_{ij}(\mathbf x):=\langle \ell_i,\pi_{\omega_i}(\dot w_0\,n(\mathbf x))e_{ij}\rangle$, where $\mathbf x=(x_\alpha,x_\beta,x_{\alpha+\beta},x_{2\alpha+\beta})=(u,v,w,z)$. A direct computation gives the root subgroup parametrizations
$$\begin{aligned}
X_\alpha(t)&=I+t(E_{12}-E_{34}),\\
X_\beta(t)&=I+tE_{23},\\
X_{\alpha+\beta}(t)&=I+t(E_{13}+E_{24}),\\
X_{2\alpha+\beta}(t)&=I+tE_{14}.
\end{aligned}$$
Thus
$$n(u,v,w,z)=X_\alpha(u)X_\beta(v)X_{\alpha+\beta}(w)X_{2\alpha+\beta}(z)
=\begin{pmatrix}
1 & u & uv+w & uw+z \\
0 & 1 & v & w \\
0 & 0 & 1 & -u \\
0 & 0 & 0 & 1
\end{pmatrix}.$$

Taking $\dot w_0=J\in N_G(A)(\mathcal O)$, and allowing reordering and individual multiplication by units in $\mathcal O^\times$, which leave the valuation conditions unchanged, we may take
$$\begin{aligned}
p_{11}&=1,\quad p_{12}=u,\quad p_{13}=uv+w,\quad p_{14}=uw+z,\\
p_{21}&=1,\quad p_{22}=v,\quad p_{23}=w,\quad p_{24}=z,\quad p_{25}=vz-w^2.
\end{aligned}$$

We apply Corollary~\ref{cor:practical} with the section $\mathcal R\subset F$ of $F/\mathcal O$ containing $0$ and with the trivial family $q_\gamma=x_\gamma$. Writing $r=(r_1,r_2)=(r,s)$, the Pl\"ucker-constrained set $S_r$ is
$$\begin{aligned}
S_r
&=\left\{(u,v,w,z)\ \middle|\ u=v=0,\quad w,z\in\mathcal R,\quad v(p_{1j})\ge-r,\quad v(p_{2j})\ge-s\right\}\\
& \stackrel{\sim}{=} \left\{(w,z)\in\mathcal R^2\ \middle|\ v(w),\,v(z)\ge-r,\quad v(w),\,v(z)\ge-s,\quad 2v(w)\ge-s\right\}\\
&=\left\{(w,z)\in\mathcal R^2\ \middle|\ v(w)\ge-\min\{r,s/2\},\quad v(z)\ge-\min\{r,s\}\right\}.
\end{aligned}$$

Finally, the number of cosets of $\mathcal O$ meeting $\{x:v(x)\ge-K\}$ is $q^K$, and $\mathcal R$ meets every coset of $\mathcal O$ in exactly one point; hence
$$|S_r|\le q^{\min\{r,s/2\}+\min\{r,s\}}\le q^{\frac34(r+s)}.$$

Suppose now that $p\neq 2$, so that $p\nmid\det(C_{\Sp_4})=2$. Then Corollary~\ref{cor:practical} shows that the Kloosterman sums for $\Sp_4$ satisfy a nontrivial bound.

\subsection{\texorpdfstring{$\SL_4$}{SL4}}

Let $G=\SL_4$, and let
$$
J=\begin{pmatrix}
0 & 0 & 0 & 1 \\
0 & 0 & -1 & 0 \\
0 & 1 & 0 & 0 \\
-1 & 0 & 0 & 0
\end{pmatrix}.
$$
Let $B=AN$ be the Borel subgroup of upper triangular matrices, with diagonal torus $A$ and unipotent radical $N$. Thus
$$A=\{\diag(a,b,c,d)\mid a,b,c,d\in F^\times,\ abcd=1\}.$$
Let $\Delta=\{\alpha_1,\alpha_2,\alpha_3\}$ be the set of simple roots. The positive roots are
$$\Phi^+=\{\alpha_1,\alpha_2,\alpha_3,\alpha_1+\alpha_2,\alpha_2+\alpha_3,\alpha_1+\alpha_2+\alpha_3\}.$$
As characters on $A$, we have
$$
\alpha_1(\diag(a,b,c,d))=ab^{-1},\qquad
\alpha_2(\diag(a,b,c,d))=bc^{-1},\qquad
\alpha_3(\diag(a,b,c,d))=cd^{-1}.
$$
Hence the coroots are
$$
\alpha_1^\vee(t)=\diag(t,t^{-1},1,1),\qquad
\alpha_2^\vee(t)=\diag(1,t,t^{-1},1),\qquad
\alpha_3^\vee(t)=\diag(1,1,t,t^{-1}),
$$
and the fundamental weights are
$$
\omega_1(\diag(a,b,c,d))=a,\qquad
\omega_2(\diag(a,b,c,d))=ab,\qquad
\omega_3(\diag(a,b,c,d))=abc.
$$

Let $V_1,V_2,V_3$ be the fundamental representations with highest weights $\omega_1,\omega_2,\omega_3$, respectively. Here $V_1$ is the $4$-dimensional standard representation, and
$$V_2=\bigwedge^2 V_1,\qquad V_3=\bigwedge^3 V_1.$$
Let $e_1,e_2,e_3,e_4$ be the standard basis of $V_1$. Their weights are
$$\omega_1,\quad \omega_1-\alpha_1,\quad \omega_1-\alpha_1-\alpha_2,\quad \omega_1-\alpha_1-\alpha_2-\alpha_3.$$
The highest weight vector of $V_1^*$ is $\ell_1:=e_4^\vee$. Hence we obtain four Pl\"ucker coordinates
$$x_j(g):=\langle ge_j,e_4^\vee\rangle=g_{4j},\qquad 1\le j\le 4.$$
For $V_2$, the highest weight vector of $V_2^*$ is $\ell_2:=e_3^\vee\wedge e_4^\vee$, and the matrix coefficients are
$$
y_{ij}(g):=\langle g\cdot(e_i\wedge e_j),\ell_2\rangle
=\det\begin{pmatrix}
g_{3i} & g_{3j} \\
g_{4i} & g_{4j}
\end{pmatrix},\qquad 1\le i<j\le 4.
$$
For $V_3$, the highest weight vector of $V_3^\vee$ is $\ell_3:=e_2^\vee\wedge e_3^\vee\wedge e_4^\vee$, and the matrix coefficients are
$$
z_{ijk}(g):=\langle g\cdot(e_i\wedge e_j\wedge e_k),\ell_3\rangle
=\det\begin{pmatrix}
g_{2i} & g_{2j} & g_{2k} \\
g_{3i} & g_{3j} & g_{3k} \\
g_{4i} & g_{4j} & g_{4k}
\end{pmatrix},\qquad 1\le i<j<k\le 4.
$$
For the computation below, it suffices to retain the following Pl\"ucker coordinates:
$$
x_2,\ x_3,\ x_4;\qquad
y_{13},\ y_{14},\ y_{23},\ y_{24},\ y_{34};\qquad
z_{124},\ z_{134},\ z_{234}.
$$

We first record the product parametrization. The root subgroup parametrizations are
$$
X_{\alpha_1}(t)=I+tE_{12},\quad
X_{\alpha_2}(t)=I+tE_{23},\quad
X_{\alpha_3}(t)=I+tE_{34},
$$
$$
X_{\alpha_1+\alpha_2}(t)=I+tE_{13},\quad
X_{\alpha_2+\alpha_3}(t)=I+tE_{24},\quad
X_{\alpha_1+\alpha_2+\alpha_3}(t)=I+tE_{14},
$$
so that, multiplying in increasing height,
$$
\prod_{\gamma\in\Phi^+}X_\gamma(x_\gamma)=
\begin{pmatrix}
1 & x_{\alpha_1} & x_{\alpha_1}x_{\alpha_2}+x_{\alpha_1+\alpha_2}
& x_{\alpha_1}x_{\alpha_2}x_{\alpha_3}+x_{\alpha_1}x_{\alpha_2+\alpha_3}+x_{\alpha_1+\alpha_2+\alpha_3}\\
0 & 1 & x_{\alpha_2} & x_{\alpha_2}x_{\alpha_3}+x_{\alpha_2+\alpha_3}\\
0 & 0 & 1 & x_{\alpha_3}\\
0 & 0 & 0 & 1
\end{pmatrix}.
$$
For the counting argument, it is convenient to change coordinates, so we set
$$
u:=x_{\alpha_1},\qquad v:=x_{\alpha_2},\qquad w:=x_{\alpha_3},
$$
$$
x:=x_{\alpha_1+\alpha_2}+x_{\alpha_1}x_{\alpha_2},\qquad
y:=x_{\alpha_2+\alpha_3}+x_{\alpha_2}x_{\alpha_3},\qquad
z:=x_{\alpha_1+\alpha_2+\alpha_3}+x_{\alpha_1}x_{\alpha_2+\alpha_3}+x_{\alpha_1}x_{\alpha_2}x_{\alpha_3}.
$$
These six polynomials are precisely the entries above the diagonal in the product above, so that in the new coordinates the product takes the simple form
$$
n(u,v,w,x,y,z)=\begin{pmatrix}
1 & u & x & z \\
0 & 1 & v & y \\
0 & 0 & 1 & w \\
0 & 0 & 0 & 1
\end{pmatrix}.
$$
Here $u,v,w$ are the coordinates of the simple roots $\alpha_1,\alpha_2,\alpha_3$, while $x,y,z$ correspond to $\alpha_1+\alpha_2,\ \alpha_2+\alpha_3,\ \alpha_1+\alpha_2+\alpha_3$, respectively. Taking $\dot w_0=J\in N_G(A)(\mathcal O)$, and allowing reordering and individual multiplication by units in $\mathcal O^\times$, we may take the Pl\"ucker coordinate polynomials to be
$$
\begin{aligned}
p_{11}&=u,\quad p_{13}=x,\quad p_{14}=z,\\
p_{21}&=v,\quad p_{22}=y,\quad p_{23}=x-uv,\quad p_{24}=z-uy,\quad p_{25}=vz-xy,\\
p_{31}&=w,\quad p_{32}=y-vw,\quad p_{33}=z-wx-uy+uvw.
\end{aligned}
$$

In the notation of \S\ref{sect:perturbation}, the six coordinate functions $u,v,w,x,y,z$ are exactly the $q_\gamma$, and each has the form $x_\gamma+f_\gamma$ with $f_\gamma\in F[x_\beta:\beta<\gamma]$. We apply Corollary~\ref{cor:practical} with the section $\mathcal R$ fixed above and with this family $q_\gamma$. Writing $r=(r_1,r_2,r_3)$, the corresponding Pl\"ucker-constrained set $S_r$ is
$$
\begin{aligned}
S_r
&=\left\{\mathbf x\in F^{\Phi^+}\ \middle|\ x_{\alpha_1}=x_{\alpha_2}=x_{\alpha_3}=0,\quad q_\gamma(\mathbf x)\in\mathcal R\ \ (\gamma\notin\Delta),\quad v(p_{ij}(\mathbf x))\ge -r_i\right\}\\
& \stackrel{\sim}{=} \left\{(u,v,w,x,y,z)\in F^6\ \middle|\ u=v=w=0,\quad x,y,z\in\mathcal R,\quad v(p_{ij})\ge -r_i\ \text{for all }i,j\right\}\\
& \stackrel{\sim}{=} \left\{(x,y,z)\in\mathcal R^3\ \middle|\ 
\begin{aligned}
&v(x)\ge-\min\{r_1,r_2\},\qquad v(y)\ge-\min\{r_2,r_3\},\\
&v(z)\ge-\min\{r_1,r_2,r_3\},\qquad v(x)+v(y)\ge-r_2
\end{aligned}
\right\},
\end{aligned}
$$
where the middle line uses $q_{\alpha_1}=u$, $q_{\alpha_2}=v$, $q_{\alpha_3}=w$, $q_{\alpha_1+\alpha_2}=x$, $q_{\alpha_2+\alpha_3}=y$, $q_{\alpha_1+\alpha_2+\alpha_3}=z$, and the last line evaluates the polynomials listed above at $u=v=w=0$, which gives $p_{11}=p_{21}=p_{31}=0$, $p_{13}=p_{23}=x$, $p_{14}=p_{24}=p_{33}=z$, $p_{22}=p_{32}=y$ and $p_{25}=-xy$.

\begin{lem}\label{lem:count-pairs}
Let $\mathcal R\subset F$ be a section of $F/\mathcal O$ containing $0$, and let $A,B,r\in\mathbb Z_{\ge0}$. Then
$$\#\left\{(x,y)\in\mathcal R^2:\ v(x)\ge-A,\quad v(y)\ge-B,\quad v(xy)\ge-r\right\}
\ \le\ q^A+q^B+(r+1)\,q^{\min\{r,\,A+B\}}.$$
\end{lem}

\begin{proof}
If $x=0$, then $xy=0$ and only $v(y)\ge-B$ is required, which holds for $q^B$ elements $y\in\mathcal R$; the case $y=0$ contributes $q^A$ pairs for the same reason.

Suppose that $x,y\neq0$. As $0$ represents the coset $\mathcal O$ in $\mathcal R$, an element of $\mathcal R$ is integral only if it is $0$; hence $v(x)=-i$, $v(y)=-j$ with $i,j\ge1$, and the conditions become $i\le A$, $j\le B$, $i+j\le r$. Since at most $q^i$ elements of $\mathcal R$ have depth $i$, and for such an $x$ at most $q^{\min\{B,\,r-i\}}$ elements $y$ are admissible, these pairs number at most
$$\sum_{i\le\min\{A,r\}}q^{i}\,q^{\min\{B,\,r-i\}}
=\sum_{i\le\min\{A,r\}}q^{\min\{i+B,\,r\}}
\le(r+1)\,q^{\min\{r,\,A+B\}} .$$
\end{proof}

By Lemma~\ref{lem:count-pairs} with $A=\min\{r_1,r_2\}$, $B=\min\{r_2,r_3\}$ and $r=r_2$, since $A,B\le r_2$, we know that the number of admissible pairs $(x,y)$ is at most
$$\ll(r_2+1)\,q^{\min\{r_2,\ \min\{r_1,r_2\}+\min\{r_2,r_3\}\}}\ \le\ (r_2+1)\,q^{\min\{r_2,\ r_1+r_3\}},$$
while, once $x$ and $y$ are fixed, $z$ has at most $q^{\min\{r_1,r_2,r_3\}}$ admissible values. Hence
$$|S_r|\ll(r_2+1)\,q^{\min\{r_2,\,r_1+r_3\}+\min\{r_1,r_2,r_3\}}
\ll q^{\frac56(r_1+r_2+r_3)},$$
the polynomial factor being absorbed by a smaller $\varepsilon$.

Suppose now that $p\neq2$, so that $p\nmid\det(C_{\SL_4})=4$. Then Corollary~\ref{cor:practical} shows that the Kloosterman sums for $\SL_4$ satisfy a nontrivial bound.

\subsection{\texorpdfstring{$G_2$}{G2}}

\subsubsection{Preparation}\label{sec:G2prep}

Let $G=G_2$, and let $\Delta=\{\alpha,\beta\}$ be the set of simple roots, where $\alpha$ denotes the short root and $\beta$ the long root. The positive roots are then
$$\Phi^+=\{\alpha,\beta,\alpha+\beta,2\alpha+\beta,3\alpha+\beta,3\alpha+2\beta\},$$
and the fundamental weights are given by
$$\omega_1=2\alpha+\beta,\qquad \omega_2=3\alpha+2\beta .$$

First, we determine the Pl\"ucker coordinates for $G_2$. Let $V_1,V_2$ be the fundamental representations with highest weights $\omega_1,\omega_2$, respectively. In this case $V_1$ is the $7$-dimensional standard representation and $V_2$ is the $14$-dimensional adjoint representation of $G_2$, so we shall obtain $21$ coordinates.

We fix a weight basis $v_1,v_2,\ldots,v_7$ of $V_1$, with respective weights
$$\omega_1,\ \alpha+\beta,\ \alpha,\ 0,\ -\alpha,\ -(\alpha+\beta),\ -\omega_1,$$
and let the lowering operators act by
$$f_1v_1=v_2,\quad f_2v_2=v_3,\quad f_1v_3=v_4,\quad f_1v_4=2v_5,\quad f_2v_5=v_6,\quad f_1v_6=v_7,$$
all remaining $f_iv_j$ being $0$. Write
$$\pi_1(g)v_j=\sum_{i=1}^7 g_{ij}(g)v_i,\qquad M(g):=(g_{ij}(g))_{1\le i,j\le 7}.$$
Since $\ell_1:=v_7^\vee$ is a highest weight vector of $V_1^*$, this gives $7$ Pl\"ucker coordinates
$$x_j(g):=\langle \pi_1(g)v_j,v_7^\vee\rangle=g_{7j}(g),\qquad 1\le j\le 7.$$

For $V_2$ we use the decomposition
$$\bigwedge^2 V_1\cong V_1\oplus V_2,$$
in which $v_1\wedge v_2$ has weight $\omega_1+(\alpha+\beta)=\omega_2$, so that $V_2$ is generated by the highest weight vector $v_1\wedge v_2$. Put $v_{ij}=v_i\wedge v_j$ and let $\ell_2:=v_6^\vee\wedge v_7^\vee$, which has weight $\omega_2$. The matrix coefficients of $V_2$ are
$$y_{ij}(g):=\langle \pi_2(g)v_{ij},\ell_2\rangle
=\det\begin{pmatrix}
g_{6i}(g) & g_{6j}(g) \\
g_{7i}(g) & g_{7j}(g)
\end{pmatrix},\qquad 1\le i<j\le 7 .$$
They span a $14$-dimensional space; the relations among them come from the $V_1$-summand of $\bigwedge^2V_1$, whose highest weight vector is $\kappa_{\omega_1}=v_{14}-2v_{23}$:
$$\begin{aligned}
    &y_{14}=2y_{23},&& 2y_{15}=y_{24}, &&2y_{16}=y_{34},\\
    &y_{35}=y_{17}+y_{26},&& 2y_{27}=y_{45}, &&2y_{37}=y_{46},\\
    &y_{47}=2y_{56}.
\end{aligned}$$
Consequently we obtain a map $N\backslash G\to F^{21}$
$$Ng\mapsto(x_1,\ldots,x_7,\ y_{12},y_{13},y_{15},y_{23},y_{25},y_{26},y_{34},y_{35},y_{36},y_{37},y_{45},y_{56},y_{57},y_{67}).$$

Second, we give a coordinate expression for $N$. A direct computation yields the parametrizations $X_\gamma:F\to N_\gamma$ of the root subgroups:
$$\begin{aligned}
    X_\alpha(t)&= I+t(E_{12}+2E_{34}+E_{45}+E_{67})+t^2E_{35},\\
    X_\beta(t)&= I+t(E_{23}+E_{56}),\\
    X_{\alpha+\beta}(t)&=I+t(E_{13}-2E_{24}+E_{46}-E_{57})-t^2 E_{26}, \\
    X_{2\alpha+\beta}(t)&= I+t(-2E_{14}+E_{25}+E_{36}-E_{47})+t^2E_{17},\\
    X_{3\alpha+\beta}(t)&= I+t(E_{15}-E_{37}),\\
    X_{3\alpha+2\beta}(t)&= I-t(E_{16}+E_{27}).
\end{aligned}$$
Hence the map $n:F^6\to N$, $(a,b,c,d,e,f)\mapsto$
$$\begin{aligned}
    &\ X_\alpha(a)X_\beta(b)X_{\alpha+\beta}(c)X_{2\alpha+\beta}(d)X_{3\alpha+\beta}(e)X_{3\alpha+2\beta}(f+cd)\\
    =\,&\left(
\begin{array}{ccccccc}
 1 & a & a b+c & -2 (a c+d) & a d+e & a b d-a c^2-f & a (-b e+c d-f)-c e+d^2 \\
 0 & 1 & b & -2 c & d & b d-c^2 & -b e+c d-f \\
 0 & 0 & 1 & 2 a & a^2 & a^2 b+2 a c+d & -a (a c+2 d)-e \\
 0 & 0 & 0 & 1 & a & a b+c & -a c-d \\
 0 & 0 & 0 & 0 & 1 & b & -c \\
 0 & 0 & 0 & 0 & 0 & 1 & a \\
 0 & 0 & 0 & 0 & 0 & 0 & 1
\end{array}
\right)
\end{aligned}$$
parametrizes $N$. In the notation of \S\ref{sect:perturbation}, the last factor being $X_{3\alpha+2\beta}(f+cd)$ means that these six coordinates are the polynomials
$$q_\alpha=x_\alpha,\ q_\beta=x_\beta,\ q_{\alpha+\beta}=x_{\alpha+\beta},\  q_{2\alpha+\beta}=x_{2\alpha+\beta},\  q_{3\alpha+\beta}=x_{3\alpha+\beta},\quad q_{3\alpha+2\beta}=x_{3\alpha+2\beta}-x_{\alpha+\beta}x_{2\alpha+\beta}.$$

Third, we compute the polynomials $p_{ij}$. For the standard Chevalley $\mathcal O$-model compatible with the root subgroup parametrizations above, choose $\dot w_0\in N_G(A)(\mathcal O)$ such that $\pi_1(\dot w_0)$ is anti-diagonal, with entries from left to right being $1,-1,1,-1,1,-1,1$. Among the seven coordinates $x_j(\dot w_0n(\mathbf x))$
the first two equal $1$ and $a$, so we keep only the remaining five:
$$p_{1,j}(\mathbf x)=x_{j+2}(\dot w_0\,n(\mathbf x)),\qquad 1\le j\le 5.$$
Likewise, among
the fourteen coordinates $y_{ij}(\dot w_0n(\mathbf x))$ listed above the first two
equal $1$ and $b$. We keep the remaining twelve.
Setting the simple-root coordinates $a=b=0$, we have
$$p_{1,1}=c,\quad p_{1,2}=-2d,\quad p_{1,3}=e,\quad p_{1,4}=-f,\quad p_{1,5}=d^2-ce,$$
$$\begin{aligned}
p_{2,1}&=d,\quad p_{2,2}=-c,\quad p_{2,3}=-e,\quad p_{2,4}=f,\quad p_{2,5}=-2c^2,\quad p_{2,6}=cd,\\
p_{2,7}&=-c^3,\quad p_{2,8}=c^2d-cf,\quad p_{2,9}=2(ce-d^2),\quad p_{2,10}=-c^2e+df,\\
p_{2,11}&=2cde-d^3-ef,\quad p_{2,12}=c^2d^2-c^3e-cdf+f^2.
\end{aligned}$$
Consequently, writing $\mathbf r=(r,s)$, the Pl\"ucker-constrained set $S_{r,s}$ of Corollary~\ref{cor:practical} is
$$
\begin{aligned}
S_{r,s}
& \stackrel{\sim}{=} \left\{(c,d,e,f)\in\mathcal R^4\ \middle|\ v(p_{1,j})\ge -r\ (1\le j\le 5),\quad v(p_{2,j})\ge -s\ (1\le j\le 12)\right\}\\
&=\left\{(c,d,e,f)\in\mathcal R^4\ \middle|\ 
\begin{aligned}
&v(c),\,v(d),\,v(e),\,v(f)\ \ge\ -\min\{r,s\},\qquad v(d^2-ce)\ \ge\ -\min\{r,s\},\\
&v(cd),\,v(c^3),\,v(c^2d-cf),\,v(c^2e-df)\ \ge\ -s,\\
&v(d^3-2cde+ef)\ \ge\ -s,\qquad v(c^2d^2-c^3e-cdf+f^2)\ \ge\ -s
\end{aligned}
\right\}.
\end{aligned}
$$

\subsubsection{Estimation}\label{sect:estimation}

We keep the notation of the previous subsection and put
$$m:=\min\{r,s\},\qquad
L_c:=\operatorname{dep}c,\quad L_d:=\operatorname{dep}d,\quad L_e:=\operatorname{dep}e,
\qquad\text{where}\quad \operatorname{dep}z:=\max\{0,-v(z)\}.$$
Since $0\in\mathcal R$ and $\mathcal R$ meets each coset of $\mathcal O$ exactly
once, we have $\operatorname{dep}z=0$ if and only if $z=0$ for $z\in\mathcal R$;
moreover $\mathcal R$ meets every coset of $\mathfrak p^{-R}$ in at most $q^R$
points.

To bound $S_{r,s}$ we fix the depths of the coordinates $c,d,e$ and count, for
each fixed choice of these depths, the number of points of $S_{r,s}$ having
them. Fix $(L_c,L_d,L_e)$. By the conditions $v(c),v(d),v(e)\ge-m$ we have
$$0\le L_c,L_d,L_e\le m,$$
so there are at most $(m+1)^3\ll (1+r+s)^3$ possible choices of
$(L_c,L_d,L_e)$. Hence it is enough to prove, uniformly in $(L_c,L_d,L_e)$,
$$\bigl|\{(c,d,e,f)\in S_{r,s}:\
(\operatorname{dep}c,\operatorname{dep}d,\operatorname{dep}e)
=(L_c,L_d,L_e)\}\bigr|\ \ll\ q^{(1-\varepsilon)(r+s)},$$
since summing over the $O((1+r+s)^3)$ possible choices of the depths produces
only a polynomial factor, which is absorbed by decreasing $\varepsilon$ if
necessary.

All the counting below is carried out with these depths fixed. That is, when we
say that a variable has at most $q^k$ choices, we always mean that, once
$(L_c,L_d,L_e)$ are fixed, there are at
most $q^k$ possibilities for it in $\mathcal R$.

Of the conditions defining $S_{r,s}$ we shall only use
$$v(c),\,v(d),\,v(e),\,v(f),\,v(d^2-ce)\ \ge\ -m,\qquad
v(c^3),\,v(cd)\ \ge\ -s,\qquad
v(d^3-2cde+ef)\ \ge\ -s,$$
together with the quadratic condition
$$f^2-(cd)\,f+c^2(d^2-ce)\ \in\ \mathfrak p^{-s} .$$

\begin{lem}[Quadratic estimate]\label{lem:quadratic}
Let $\gamma,\delta\in F$ and $R\ge0$. Then
$$\#\{\,z\in\mathcal R\ :\ z^2+\gamma z+\delta\in\mathfrak p^{-R}\,\}\ \ll\ q^{R/2},$$
with an implied constant depending only on $F$.
\end{lem}
\begin{proof}
If the set in question is empty, there is nothing to prove. Otherwise, fix $z_0\in\mathcal R$ such that $z_0^2+\gamma z_0+\delta\in\mathfrak p^{-R}$. For any other such $z\in\mathcal R$, we have
\[
(z-z_0)(z+z_0+\gamma)=(z^2+\gamma z+\delta)-(z_0^2+\gamma z_0+\delta)\in\mathfrak p^{-R}.
\]
Put $a=\lfloor R/2\rfloor$. Since valuations take integer values, at least one of the two factors has valuation at least $-a$. Consequently, every such $z$ belongs to
\[
\bigl(z_0+\mathfrak p^{-a}\bigr)\cup\bigl(-z_0-\gamma+\mathfrak p^{-a}\bigr).
\]
Each of these two cosets is a union of $q^a$ cosets of $\mathcal O$, and $\mathcal R$ contains at most one representative of each coset of $\mathcal O$. Therefore the number of possible $z$ is at most $2q^a\le 2q^{R/2}$. This proves the lemma.
\end{proof}

The counting is carried out in the order $c,d,e,f$, each step being controlled
by one of the conditions listed above. Then we obtain the following bound for positive depths.

\begin{prop}[Positive depths]\label{prop:positive-depth}
Assume that $L_c,L_d,L_e\ge1$, and put $m:=\min\{r,s\}$ and
$$E:=L_c+L_d+\min\{L_e,\,m-L_c\}+\min\Bigl\{s-L_e,\,\frac s2,\,m\Bigr\}.$$
Then the number of points of $S_{r,s}$ with $\operatorname{dep}c=L_c$,
$\operatorname{dep}d=L_d$ and $\operatorname{dep}e=L_e$ is $\ll q^{E}$, and
furthermore $E\le\frac{21}{22}(r+s)$.
\end{prop}
\begin{proof}
Fix the depths $L_c,L_d,L_e\ge1$. Recall that such depths satisfy
$$
L_c\le\min\Bigl\{m,\frac s3\Bigr\},\quad L_d,\,L_e\le m,\quad
L_d>\frac m2\ \Longrightarrow\ L_e=2L_d-L_c :
$$
the first two bounds follow from $v(c),v(d),v(e)\ge-m$ and $v(c^3)\ge-s$, and the
implication follows from $v(d^2-ce)\ge-m$. In particular, $L_c,L_d,L_e\le m$, and $c$ and
$d$ have at most $q^{L_c}$ and $q^{L_d}$ choices, respectively.

The condition $v(d^2-ce)\ge-m$ says that $ce\in d^2+\fp^{-m}$, i.e.\ that $e$ is
determined modulo $\fp^{-(m-L_c)}$; hence $e$ has at most
$q^{\min\{L_e,\,m-L_c\}}$ choices. Finally, $f$ satisfies
$$v(f)\ge-m,\quad ef\in(2cde-d^3)+\fp^{-s},\quad f^2-(cd)f+c^2(d^2-ce)\in\fp^{-s},$$
and these three conditions bound the number of choices for $f$ by
$$\ll\min\bigl\{q^{\,m},\ q^{\,s-L_e},\ q^{\,s/2}\bigr\},$$
using Lemma~\ref{lem:quadratic}. Multiplying the four estimates, we see
that the number of points of $S_{r,s}$ with the given depths is $\ll q^{E}$.

It remains to show that $E\le\frac{21}{22}(r+s)$. Put $t:=\min\{m,\frac s3\}$ and
$K:=\min\{m,\frac s2\}$. If $L_d\le\frac m2$, then for fixed $L_c$ the function
$L_e\mapsto\min\{L_e,m-L_c\}+\min\{s-L_e,K\}$ is a tent with maximum
$\min\{m-L_c+K,s\}$, so
$$E\le B_1:=\frac m2+\min\{m+K,\,s+t\}\le\frac{21}{22}(m+s)\le\frac{21}{22}(r+s).$$
If $L_d>\frac m2$, then $L_e=2L_d-L_c$ and $\min\{L_e,m-L_c\}=m-L_c$; writing
$L_d=\frac m2+\theta$ with $0<\theta\le\frac{L_c}2$, this becomes
$$E=\frac{3m}2+\min\{s-m-\theta+L_c,\ K+\theta\}.$$
The first entry decreases and the second increases in $\theta$, so the supremum is
obtained as $\theta\to0$, at $\theta=\frac{L_c}2$, or at the crossing, and
consequently $E\le\max\{B_1,B_2,B_3\}$, where
$$B_2:=\frac{3m}2+\min\{s-m,\,K\}+\frac t2,\qquad
B_3:=m+\frac{s+K}2+\frac t2 .$$
Here $B_1$ is the bound of the first case, and $B_2,B_3\le\frac{21}{22}(m+s)$ as
well: both are elementary one-variable estimates in $s/m$ (split according as
$s\ge3m$, $2m\le s\le3m$, or $s\le2m$), and for $B_3$ one uses in addition that
the crossing lies in the interval only when $s\ge\frac65m$. Since $m\le r$, we get
$E\le\frac{21}{22}(r+s)$ in all cases.
\end{proof}

The cases where one of the depths vanishes are easier.

\begin{prop}[Zero depth]\label{prop:zero-depth}
Assume that $L_cL_dL_e=0$. Then the number of $(c,d,e,f)\in S_{r,s}$ with these
depths is
$$\ll\ q^{\frac{5}{6}(m+s)} .$$
\end{prop}
\begin{proof}
Since $\operatorname{dep}z=0$ means $z=0$ for $z\in\CR$, at least one of $c,d,e$
vanishes and contributes a single class. Recall that a coset of $\fp^{-R}$ meets
$\CR$ in at most $q^{R}$ classes, that $v(f)\ge-m$, and that
$f^2-(cd)f+c^2(d^2-ce)\in\fp^{-s}$ allows $\ll q^{s/2}$ possible values of $f$ by
Lemma~\ref{lem:quadratic}.

If $c=0$, then $v(d^{2})\ge-m$, so $L_d\le\frac m2$; $e$ obeys only
$v(e)\ge-m$, while $d^{3}+ef\in\fp^{-s}$ puts $f$ in a coset of $\fp^{-(s-L_e)}$
when $e\ne0$ (if $e=0$ the term $q^{s-L_e}=q^{s}$ is dominated by $q^{m}$). Hence
the number of classes is bounded by a constant times
$$q^{\frac m2+L_e+\min\{m,\frac s2,s-L_e\}}\le q^{\frac m2+\min\{s,2m\}}
\le q^{\frac56(m+s)} .$$

If $c\ne0$ and $d=0$, then $v(ce)\ge-m$, i.e.\ $L_c+L_e\le m$, and the bound for
$f$ bounds the number of classes by a constant times
$$q^{L_c+L_e+\min\{m,\frac s2\}}\le q^{m+\min\{m,\frac s2\}}\le q^{\frac34(m+s)} .$$

If $c,d\ne0$ and $e=0$, then $v(c^{3}),v(d^{3})\ge-s$, i.e.\
$L_c,L_d\le\min\{m,\frac s3\}$, and the bound for $f$ bounds the number of classes by a constant times
$$q^{2\min\{m,\frac s3\}+\min\{m,\frac s2\}}\le q^{\frac79(m+s)} .$$

Combining these three cases gives the asserted bound.
\end{proof}

Combining Propositions
\ref{prop:positive-depth} and \ref{prop:zero-depth}, we obtain the following bound for $|S_{r,s}|$.

\begin{cor}\label{cor:G2-count}
$$|S_{r,s}|\ \ll\ (1+r+s)^3\,q^{\frac{21}{22}(r+s)} .$$
\end{cor}

\subsubsection{Conclusion}

\begin{thm}\label{thm:G2-power}
Let $G=G_2$. Then the Kloosterman sums $\{\Kl_d\}_{d\ge1}$ on $G$ satisfy a
nontrivial bound: for all sufficiently large $d$ there exist $\varepsilon>0$ and
$C>0$ such that
$$\bigl|\Kl_d(t)\,\delta^{1/2-\varepsilon}(t)\bigr|\le C,\qquad t\in A .$$
\end{thm}
\begin{proof}
By Corollary~\ref{cor:G2-count} there are constants $\varepsilon_0>0$ and $C>0$
such that
$$|S_{r,s}|\le C\,q^{(1-\varepsilon_0)(r+s)},\qquad r,s\in\mathbb Z_{\ge0}.$$
Since $\det(C_{G_2})=1$,
the assumption $p\nmid\det(C_G)$ of Theorem~\ref{main-thm} holds for every $p$.
Theorem~\ref{main-thm} therefore applies.
\end{proof}

\smallskip
As an application, we record the regularity of Bessel distributions on $G_2$; we
refer to Section \ref{sec:intro} for the definition of the Bessel distribution $B_\pi$. Let $\psi = \psi_0$ be the generic character
on $N$ in \cite[Theorem 1.2]{CCL26}.

\begin{thm}\label{thm:G2-Bessel}
Let $G=G_2$ and let $\pi$ be an irreducible smooth admissible generic
representation of $G$ with respect to $\psi$. Then the Bessel distribution
$B_\pi$ is regular: there exist a dense subset $\Omega\subset G$ and a smooth
function $j_\pi$ on $\Omega$, locally integrable on $G$, such that
$$B_\pi(f)=\int_G j_\pi(g)\,f(g)\,dg,\qquad f\in\CS(G).$$
The function $j_\pi$ is unique.
\end{thm}
\begin{proof}
By \cite[Theorem 1.2]{CCL26}, regularity holds if Kloosterman sums
on each standard Levi subgroup of $G = G_2$ have nontrivial bounds.
For $G = G_2$, standard Levi subgroups are isomorphic to the maximal torus $A$, $\GL_2$ and $G$. 
Since we already have the Weil bound for Kloosterman sums on $\GL_2$
(See for example \cite[Section 4.1]{CCL26}), the regularity follows from the nontrivial bound on $G$.
\end{proof}

\end{document}